\documentclass{article}
\usepackage[utf8]{inputenc}
\usepackage{amsmath,amssymb,latexsym,amsthm,mathrsfs,bbm}
\usepackage{hyperref}
\usepackage{color}
\usepackage[dvipsnames]{xcolor}
\usepackage{cleveref}
\usepackage{appendix}
\usepackage{float}
\usepackage{soul}
\usepackage{enumitem}
\usepackage{xparse}
\usepackage{tikz}
\usepackage{tikz-3dplot}
\usepackage{mathdots}
\usepackage{yhmath}
\usepackage{cancel}
\usepackage{siunitx}
\usepackage{array}
\usepackage{multirow}
\usepackage{gensymb}
\usepackage{tabularx}
\usepackage{booktabs}
\usepackage{tkz-euclide}
\usetikzlibrary{fadings}
\usetikzlibrary{patterns}
\usetikzlibrary{shadows.blur}
\usetikzlibrary{shapes}
\usetikzlibrary{arrows.meta}
\usetikzlibrary{backgrounds,calc,positioning}
\usetikzlibrary{angles,quotes}
\newtheorem{theorem}{Theorem}[section]

\newtheorem{lemma}[theorem]{Lemma}
\newtheorem{proposition}[theorem]{Proposition}

\newtheorem{remark}{Remark}

\newcommand{\R}{\mathbb{R}}

\newcommand{\Rd}{\mathbb{R}^d}

\newcommand{\iRd}{\int_{\mathbb{R}^d}}

\newcommand{\japangle}[1]{\left\langle #1\right\rangle}

\newcommand{\dvn}{\nabla_v}
\newcommand{\dvnstar}{\nabla_{v_*}}
\newcommand{\rd}{\mathrm{d}}

\title{Convergence of a particle method for a regularized spatially homogeneous Landau equation}
\author{Jos\'e A. Carrillo\thanks{Mathematical Institute, University of Oxford, Oxford OX2 6GG, UK (carrillo@maths.ox.ac.uk)}, Matias G. Delgadino\thanks{Department of Mathematics, The University of Texas at Austin, Texas, USA (matias.delgadino@math.utexas.edu)}, Jeremy S. H. Wu\thanks{Mathematical Sciences Building, University of California, Los Angeles, USA (jeremywu@math.ucla.edu)}}

\begin{document}
\maketitle

\begin{abstract}
	We study a regularized version of the Landau equation, which was recently introduced in~\cite{CHWW20} to numerically approximate the Landau equation with good accuracy at reasonable computational cost. We develop the existence and uniqueness theory for weak solutions, and we reinforce the numerical findings in~\cite{CHWW20} by rigorously proving the validity of particle approximations to the regularized Landau equation.
\end{abstract}

\section{Introduction} 
The Landau equation \cite{Landau}, originally derived to approximate the Boltzmann operator when collisions between charged particles in a plasma are grazing, is one of the fundamental kinetic equations in plasma physics. Efficient computational methods for the full Vlasov-Maxwell-Landau system are of tremendous importance for modelling future fusion reactors and they represent a central conundrum in computational plasma physics. An important foundation to achieve such an ambitious goal is to provide accurate numerical methods with low computational cost to solve the collisional step in these computations, that is, to solve the spatially homogeneous Landau equation given by
\begin{equation} 
	\label{landau}
	\partial_t f=Q(f,f):= \nabla_v \cdot \left\{ \int_{\R^d} A(v-v_*) \left( f(v_*) \dvn f(v) - f(v) \dvnstar f(v_*) \right) \rd v_*\right\}\,,
\end{equation}
with the collision kernel given by $A(z) = |z|^{\gamma} \left(|z|^2 I_d - z \otimes z \right)= |z|^{\gamma+2} \Pi (z)$ with $I_d$ being the identity matrix, $\Pi(z)$ the projection matrix into $\{z\}^\perp$, $-d-1\leq \gamma \leq 1$, and $d\geq 2$. The most important case corresponds to $d=3$ with $\gamma = -3$ associated with the physical interaction in plasmas. This case is usually called the Coulomb case because it can be derived from the Boltzmann equation in the grazing collision limit when particles interact via Coulomb forces \cite{DLD92,Vi98,CDW22}. The main formal properties of $Q$ rely on the following reformulation  
\begin{equation*}
	Q(f,f) = \dvn \cdot  \left\{ \int_{\R^d} A(v-v_*) f f_* \left (  \dvn \log f - \dvnstar \log f_* \right)\,  \rd v_*\right\}\,,
\end{equation*}
where $f=f(v)$, $f_*=f(v_*)$ are used; and its weak form acting on appropriate test functions $\phi = \phi(v)$ 
\begin{equation} \label{weak}
	\int_{\R^d} Q(f,f) \phi \,\rd{v}= -\frac12 \iint_{\R^{2d}} (\dvn \phi - \dvnstar \phi_*)  \cdot  A(v-v_*) \left(  \dvn \log f - \dvnstar \log f_*  \right) f f_* \, \rd v \,\rd v_*\,.
\end{equation}
Then choosing $\phi(v) = 1, v, |v|^2$, one achieves conservation of mass, momentum and energy. Inserting $\phi(v) = \log f (v)$, one obtains the formal entropy decay with dissipation given by
\begin{equation*}
	\frac{\rd}{\rd t}\int_{\mathbb{R}^d}f \log f \,\rd{v} = -D(f(t,\cdot)):= -\frac{1}{2}\iint_{\R^{2d}} B_{v,v_*} \cdot A(v-v_*)B_{v,v_*}ff_* \, \rd v\rd v_*\leq 0\,,
\end{equation*}
since $A$ is symmetric and semipositive definite, with $B_{v,v_*} := \dvn \log f - \dvnstar \log f_*$. The equilibrium distributions are given by the Maxwellian 
\[
\mathcal{M}_{\rho, u, T} = \frac{\rho}{(2\pi T)^{d/2}} \text{exp} \left(- \frac{ |v-u|^2}{2T} \right),
\]
for some constants $\rho, T$ determining the density and the temperature of the particle ensemble, and mean velocity vector $u$, see \cite{Vi98,GZ}.

Deterministic numerical methods based on particle approximations to \eqref{landau} have been recently proposed in \cite{CHWW20} keeping all the structural properties of the Landau equation described above: nonnegativity, conservation of mass, momentum and energy, and entropy dissipation at a semidiscrete level. This paper gives a theoretical underpinning to the numerical sheme introduced in \cite{CHWW20}. The main strategy is to delocalize the gradient operators in the weak form \eqref{weak} while keeping intact the variational structure behind the equation rigorously developed in \cite{CDDW20}. This is reminiscent of similar approaches to approximate nonlinear diffusion models by nonlocal equations \cite{CCP18} while keeping their variational structure. More precisely, we analyse the Landau gradient flow of the regularized entropy~\cite{CHWW20,CDDW20} given by
\begin{equation}
	\label{eq:epslan}
	\partial_t f = \nabla_v \cdot \left\{
	f(v) \iRd f(w) A(v-w)(\nabla G^\varepsilon * \log [f*G^\varepsilon](v) - \nabla G^\varepsilon * \log [f*G^\varepsilon](w)) \, \rd w
	\right\},
\end{equation}
where $ G^\varepsilon\in C^\infty(\R^d)$ is a mollifier for fixed $\varepsilon>0$. More specifically, 
\[
G^\varepsilon(v) = \frac{1}{\varepsilon^d} G\left(
\frac{v}{\varepsilon}
\right), \quad \int_{\R^d} G(v)dv = 1,
\]
with $0\le G\in C^\infty(\R^d)$, so that $G^\varepsilon$ approximates the Dirac at the origin, $\delta_0$, as $\varepsilon \downarrow 0$. Therefore, as $\varepsilon \downarrow 0$ \eqref{eq:epslan} formally converges to the Landau equation. For technical reasons (c.f.~\Cref{lem:difflog}), we choose $G(v) = Ce^{-(1+|v|^2)^{1/2}}$ as in \cite{CDDW20}. However, we note that from the numerical point of view~\cite{CHWW20}, Gaussian mollifiers are simpler to deal with. 

Our approach is to provide an existence theory for \eqref{eq:epslan} as well as a particle approximation to the solution by interpreting~\eqref{eq:epslan} as a continuity equation with solution-dependent velocity fields. In particular, to introduce notation, we define a generalised interaction kernel for probability measures $g\in \mathscr{P}(\R^d)$ and $v, w \in \R^d$
\[
K_g(v,w) := - |v-w|^{2+\gamma}\Pi[v-w] (\nabla G^\varepsilon * \log [g*G^\varepsilon](v) - \nabla G^\varepsilon * \log [g*G^\varepsilon](w)).
\]
Additionally, for $f\in\mathscr{P}(\R^d)$, we define the measure-dependent velocity
\[
U^\varepsilon[g,f](v) := \int_{\R^d}K_g(v,w) \, \rd f(w), \quad U^\varepsilon[f] := U^\varepsilon[f,f].
\]
In this way, \eqref{eq:epslan} can be written as
\begin{equation}
	\label{eq:epscty}
	\partial_t f + \nabla \cdot (U^\varepsilon[f] f) = 0.
\end{equation}
Formally speaking, by approximating an initial data by a finite number of atomic measures, we expect the solution of \eqref{eq:epscty} to be approximated by a finite number of Dirac masses following the local velocity of particles. More precisely, suppose we are given initial data $f^0 \in \mathscr{P}(\R^d)$ for~\eqref{eq:epslan} and we can approximate $f^0$ by a sequence of empirical measures $\mu_0^N = \frac{1}{N} \sum_{i=1}^N \delta_{v_0^i}$, with equal weights for simplicity, where $v_0^i \in \R^d$ for $i=1,\dots, N$. We expect the solution to \eqref{eq:epscty} to be given by the empirical measure with equal weights
$$
\mu^N(t)= \frac1N \sum_{i = 1}^N \delta_{v^i(t)}\,,
$$
where $v^i(t)$ is the solution of the ODE system
$$
\dot{v}^i(t)=U[\mu^N(t)](v^i(t)).
$$ 
In fact, $\mu^N(t)$ is a distributional solution to \eqref{eq:epscty} with $\mu_0^N$ as initial data. The results in \cite{CDDW20} do not provide a well-posedness of measure solutions to \eqref{eq:epslan} with measure initial data, ensuring only the existence by compactness. Due to the lack of continuous dependence with respect to initial data to \eqref{eq:epscty} in the general probability measure setting, showing the convergence of the mean field limit is important from the numerical viewpoint \cite{CHWW20}. More specifically, we show that $\mu^N$ converges towards the unique weak solution $f$ of \eqref{eq:epscty} in the limit $N\to \infty$. 

In the simplified setting of equally weighted particles, the main result of this paper can be summarized as follows. Suppose that the initial data are well approximated in the sense of
\[
W_\infty(\mu_0^N, \, f^0) \to 0 \quad\mbox{as $N\to \infty$ fast enough,}
\]
where $W_\infty$ denotes the $\infty$-Wasserstein metric~\cite{RR98}, see Hypothesis~\ref{B:initmf} and \ref{B:initcond}. As mentioned above, the evolution of $\mu_0^N$ through~\eqref{eq:epslan} is characterised entirely by the evolution of the `particles' starting at $v_0^i$ according to the ODE system for $i=1,\dots, N$
\begin{align}\label{eq:particlesys}
	\dot{v}^i &= - \frac{1}{N}\sum_{j=1}^N |v^i - v^j|^{2+\gamma}\Pi[v^i - v^j](\nabla G^\varepsilon * \log [\mu^N * G^\varepsilon](v^i) - \nabla G^\varepsilon * \log [\mu^N * G^\varepsilon](v^j)).
\end{align}
We will prove, at least for short times depending on the value of $\gamma\in(-3,0]$, that $f = f(t) = f_t$ and $\mu^N = \mu^N(t) = \mu_t^N$ exist (c.f.~\Cref{thm:existmf} and~\Cref{lem:existparticle}, respectively) and solve~\eqref{eq:epslan} according to the initial conditions $f^0$ and $\mu_0^N$, respectively. Given the existence of such curves $f, \, \mu^N$ and the fact that $\mu_0^N \to f^0$ as $N\to \infty$, we seek to prove the mean field limit (c.f.~\Cref{thm:CCH})
\[
W_\infty(\mu^N(t), \, f(t)) \to 0, \quad \text{for }t\in (0,T_m) \text{ as }N\to \infty,
\]
where $T_m>0$ is the maximal existence time of $f$ (c.f.~\Cref{thm:existmf}). 

The mean-field limit has attracted lots of attention in the last years in different settings for aggregation-diffusion and Vlasov type kinetic equations. Different approaches have been taken leading to a very lively interaction between different communities of researchers in analysis and probability. We refer to \cite{BH,dobru,Neun,Spohn,G03} for the classical approaches in the field. Recent advances in non-Lipschitz settings and with applications to models with alignment have been done in \cite{BCC11,CS18,CCHS19,CS19}, for the aggregation-diffusion and Vlasov-type equations in  \cite{HJ07,CDFLS11,CCH14,J14,H14,HJ15,Due16,G16,JW16,LP17,JW17,PS17,JW18,BJW19,S20,BJS22}, and for incompressible fluid problems \cite{H09,HIpre}. 

We prove the mean field limit to the regularized Landau equation \eqref{eq:epslan} following the strategy and ideas from~\cite{H09,CCH14}. The main difference with these references is the fact that equation \eqref{eq:epslan} is more nonlocal, and it can be interpreted as a transport equation with a highly nonlocal nonlinear mobility depending quadratically on the density $f$. Let us finally mention that our result does not give quantitative bounds on the mean-field limit depending on $N$ and $\varepsilon$ compared to recent works \cite{JW16,JW17,JW18,BJW19}. This is certainly an important open question of great importance from the numerical viewpoint.
\subsection{Main results}
The proof of the mean field limit convergence of~\eqref{eq:epslan} for fixed $\varepsilon>0$ is achieved with the following strategy borrowed from \cite{CCH14}: we first show the existence and uniqueness of the continuity equation~\eqref{eq:epscty} for some maximal time horizon $T_m>0$ in~\Cref{sec:estvel,sec:existcts}, then we show that the particle system does also exist in~\Cref{sec:IPS}. We finally conclude by estimating the distance between the two systems in the $W_\infty$ metric when $N\to\infty$ as well as establishing a lower bound on the existence of the particle system~\Cref{sec:mfl}. Let us point out that since the kernel $A$ is singular or grows at infinity, these properties of the continuity equation and the associated particle system are not obvious.

Continuity equations of the form~\eqref{eq:epscty} have been extensively studied~\cite{G03,BLR11}. To obtain well-posedness, we show regularity and growth estimates on $U^\varepsilon[f]$, that stem from the following regularity assumptions on the initial data. 
\begin{enumerate}[label=\textbf{A\arabic*}]
	\item \label{A:cpctsupp} The initial condition $f^0$ belongs to $\mathscr{P}_c(\R^d)$, the space of compactly supported probability measures on $\R^d$.
	\item \label{A:gamma} For $\gamma\in (-3,-2)$, there exists $p>1$ such that $\frac{p}{p-1}(2+\gamma)>-d$ and $f^0$ belongs to $L^p(\R^d)$.
\end{enumerate}
\begin{theorem}[Existence of mean field limit]
	\label{thm:existmf}
	Fix $\varepsilon>0$, $\gamma\in(-3,0]$, and initial data $f^0\in \mathscr{P}(\R^d)$ satisfying~\ref{A:cpctsupp} and~\ref{A:gamma}.
	Then, there is a time horizon $T = T(\gamma,\varepsilon, f^0)>0$ such that there is a unique weak solution $f$ to~\eqref{eq:epslan} given in
	\[
	f \in \left\{
	\begin{array}{cl}
		C([0,T]; \mathscr{P}_c(\R^d)), 	&\gamma\in[-2,0] 	\\
		C([0,T]; \mathscr{P}_c(\R^d)) \cap L^\infty(0,T; L^p(\R^d)), 	&\gamma\in(-3,-2)
	\end{array}
	\right.,
	\]
	where $f|_{t=0} = f^0$, and the exponent $p>1$ is the same as in~\ref{A:gamma}. 
	
	In the case $\gamma\in[-2,0]$, the maximal time of existence $T_M=+\infty$ is infinite. While for the case $\gamma\in(-3,-2)$, either  the maximal time of existence is infinite $T_M=+\infty$, or the $L^p$ norm of the solution blows up
	\[
	\mathrm{esssup}_{s\in[0,t)}\|f(s)\|_{L^p}\uparrow + \infty \quad \text{as}\quad t \uparrow T_M.
	\]
\end{theorem}
The notion of weak solution $f$ to~\eqref{eq:epscty} (equivalently~\eqref{eq:epslan}) means that, for any $\phi\in C_c^\infty(\R^d)$, the following equality holds
\[
\frac{\rd}{\rd t}\int_{\R^d} \phi(v) \, \rd f_t(v) = \int_{\R^d} \nabla \phi(v) \cdot U^\varepsilon[f(t)](v) \, \rd f_t(v).
\]
The particle solution \eqref{eq:particlesys} is in fact a solution to the previous equation when the initial condition is a convex combinations of delta measures~\eqref{eq:epscty}. More precisely, for every $N\in\mathbb{N}$, take initial points $\{v_0^{i,N}\}_{i=1}^N\subset \R^d$ and positive weights $\{m_{i,N}\}_{i=1}^N$ satisfying
\[
\sum_{i=1}^N m_{i,N} = 1, \quad m_{i,N} \ge 0, \quad \forall i=1,\dots,N.
\]
The $N$-particle ODE system we consider is
\begin{align}
	\label{eq:epsparticle}
	\begin{split}
		\frac{\rd}{\rd t}v^i(t) &= U^\varepsilon[\mu^N(t)](v^i) = \sum_{j=1}^Nm_j K_{\mu^N(t)}(v^i,\, v^j), \\
		\mu^N(t) &= \sum_{i=1}^N m_i \delta_{v^i(t)}, \\
		\left.v^i\right|_{t=0} &= v_0^i.
	\end{split}
\end{align}
\begin{lemma}[Existence of particle solutions]
	\label{lem:existparticle}
	For any $\varepsilon>0$, and $\gamma\in(-3,0]$, there exists a time horizon $T = T\left(\varepsilon, \gamma, \left\{v_0^{i,N} \right\}_{i=1}^N \right)>0$ and a curve  $v^{i,N} \in C^1([0,T];\R^d)$ which satisfies ~\eqref{eq:epsparticle}. For $\gamma\in[-2,0]$, the solution to \eqref{eq:epsparticle} is unique, and the time horizon $T$ can be arbitrarily large.
\end{lemma}
The well-posedness of the inter-particle system~\eqref{eq:epsparticle} is proven in~\Cref{sec:IPS}. The case $\gamma\in[-2,0]$ is an application of~\Cref{thm:existmf}, while for $\gamma\in(-3,-2)$ a standard Peano existence argument is used. 

For $N\in \mathbb{N}$ and trajectories $\{v^{i,N}(t)\}_{i=1}^N$ such as those constructed in~\Cref{lem:existparticle}, we define the minimum inter-particle distance for times $t$ in the domain of existence
\[
\eta_m^N(t) := \min_{i\neq j} |v^{i,N}(t) - v^{j,N}(t)|.
\]
Taking the continuum and particle solutions $f$ and $\mu^N$ from \Cref{thm:existmf} and \Cref{lem:existparticle} respectively, we define
\[
\eta^N(t) := W_\infty(\mu^N(t), f(t)).
\]
The following assumptions are well-preparedness conditions on the initial data of the particle solution, see \cite{CCH14}.

\begin{enumerate}[label=\textbf{B\arabic*}]
	\item \label{B:initmf} The initial particles $\{v_0^{i,N}\}_{i=1}^N \subset \R^d$ and weights $\{m_{i,N}\}_{i=1}^N\subset (0,1)$ satisfy $W_\infty(\mu_0^N, f^0)\to 0$ as $N\to \infty$.
	\item \label{B:initcond} For $\gamma\in(-3,-2)$, the initial particles moreover satisfy
	\begin{equation}
		\label{eq:initcond}
		\lim_{N\to \infty} \eta^N(0)^\frac{d}{p'}\eta_m^N(0)^{1+\gamma} = 0,
	\end{equation}
	where the conjugate exponent $p'$ satisfies $\frac{1}{p}+\frac{1}{p'}=1$ with $p$ from~\ref{A:gamma}.
\end{enumerate}
The main result concerning the mean field limit can now be stated as:

\begin{theorem}
	\label{thm:CCH}
	Fix $\varepsilon>0$, $\gamma\in(-3,0]$, and initial data $f^0$ satisfying~\ref{A:cpctsupp} and~\ref{A:gamma}. We consider $f$ the solution to~\eqref{eq:epslan} on the maximal time interval $[0,T_m]$ provided by~\Cref{thm:existmf}. Given initial particle configurations $\{\mu_0^N\}_{n\in\mathbb{N}}$ satisfying~\ref{B:initmf} and~\ref{B:initcond}, we consider $\{\mu^N\}_{N\in\mathbb{N}}$ particle solutions of~\eqref{eq:epslan}  with maximal time of existence $T^N>0$ provided by Lemma~\ref{lem:existparticle}. Then $\liminf_{N\to\infty}T^N \ge T_m$, and the mean field limit holds
	\begin{equation}
		\label{eq:mfl}
		\lim_{N\to\infty}\sup_{t\in[0,T]}W_\infty(\mu^N(t), f(t))\to 0, \quad \forall T\in[0,T_m).
	\end{equation}
\end{theorem}

\section{Estimates on the velocity}
\label{sec:estvel}
This section collects the necessary estimates on the measure-dependent kernel and velocity, $K$ and $U^\varepsilon$. To fix notation, we define the Lebesgue bracket
\[
\japangle{v}^2 = 1+|v|^2, \quad v\in\R^d,
\]
and the mollifying sequence by
\[
G(v) = Ce^{-\japangle{v}}, \quad \int_{\R^d}G(v) \rd v = 1, \quad G^\varepsilon(v) = \frac{1}{\varepsilon^d}G\left(v/\varepsilon\right).
\]
Moreover, we define the $p$th order moment of a measure $f$ by
\[
M_p(f) = \int_{\R^d} \japangle{v}^p \rd f(v).
\]
We will use the notation $a\le_{\alpha,\beta,\dots} b$ to represent the statement that there is a constant $C=C(\alpha,\beta,\dots)$ such that $a\le C b$.
\begin{proposition}
	\label{prop:LipKg}
	Fix $\varepsilon>0$ and $\gamma=-2$. Then, for every $f,\, g\in \mathscr{P}(\R^d)$, the functions $K_g(v,w)$ and $U^\varepsilon[g,f](v)$ are $C^1$, skew-symmetric and satify the estimates
	\[
	|K_g(v,w)| \le \frac{2}{\varepsilon}, \quad |\nabla_v K_g(v,w)| \le \frac{28}{\varepsilon^2},
	\]
	\[
	|U^\varepsilon[g,f](v)| \le \frac{2}{\varepsilon}, \quad |\nabla U^\varepsilon[g,f](v)| \le \frac{28}{\varepsilon^2}.
	\]
\end{proposition}
\Cref{prop:LipKg} highlights the $C^1$-boundedness of the velocity field $U^\varepsilon$ in the special case $\gamma=-2$. For this value of $\gamma$, the well-posedness of~\eqref{eq:epscty} follows by standard techniques~\cite{G03}. \Cref{prop:LipKg} follows from the more general results~\Cref{prop:boundsKg}, \Cref{lem:contourestimate}, and~\Cref{lem:UHold} where $\gamma\in[-3,0]$. There, we shall see the precise dependence on $\gamma$. First, we recall a standard inequality for the Lebesgue bracket. 
\begin{lemma}[Peetre]
	\label{lem:peetre}
	For any $p\in\mathbb{R}$ and $x,y\in\Rd$, we have
	\[
	\frac{\langle x\rangle^p}{\langle y\rangle^p} \leq 2^{|p|/2} \langle x-y\rangle^{|p|}.
	\]
\end{lemma}
\begin{proof}
	A proof of this can be found in~\cite[Lemma 43]{CDDW20} or~\cite{B73}.
\end{proof}
\begin{lemma}[log-derivative estimates]
	\label{lem:estlogdiff}
	For fixed $\varepsilon>0$ we have the formula
	\begin{equation}
		\label{eq:diffGseps}
		\nabla G^\varepsilon(v) = -\frac{1}{\varepsilon}\japangle{\frac{v}{\varepsilon}}^{-1}G^\varepsilon(v) \frac{v}{\varepsilon}.
	\end{equation}
	For $\mu\in\mathscr{P}(\Rd)$, denoting $\partial^i = \frac{\partial}{\partial v^i}$ and $\partial^{ij} = \frac{\partial^2}{\partial{v^i}\partial{v^j}}$, we obtain
	\begin{equation}
		\label{eq:extlogdiffsgeq1}
		\left|
		\nabla \log (\mu*G^\varepsilon)(v)
		\right|\leq \frac{1}{\varepsilon}, \quad \left|
		\partial^{ij} \log (\mu*G^\varepsilon)(v)
		\right|\leq \frac{4}{\varepsilon^2}.
	\end{equation}
\end{lemma}
\begin{proof}
	This is proven in~\cite[Lemma 30]{CDDW20}.
\end{proof}

\begin{proposition}
	\label{prop:boundsKg}
	Fix $\varepsilon>0,\, g \in \mathscr{P}(\R^d)$, and $\gamma\in[-3,0]$. We have the following estimate
	\[
	|K_g(v,w)| \le \min\left(
	\frac{4}{\varepsilon^2}|v-w|^{3+\gamma},\frac{2}{\varepsilon}|v-w|^{2+\gamma}
	\right).
	\]
	Moreover, for fixed $f\in\mathscr{P}(\R^d)$, we have
	\[
	|U^\varepsilon[g,f](v)|\lesssim_\varepsilon\left\{
	\begin{array}{cl}
		M_{2+\gamma}(f)\japangle{v}^{2+\gamma}, 	&\gamma\in(-2,0] 	\\
		1, 	&\gamma\in[-3,-2]
	\end{array}
	\right..
	\]
\end{proposition}
\begin{proof}
	We recall the expression
	\[
	K_g(v,w) = - |v-w|^{2+\gamma}\Pi[v-w](\nabla G^\varepsilon * \log [g*G^\varepsilon](v) - \nabla G^\varepsilon * \log [g*G^\varepsilon](w)).
	\]
	\underline{$|v-w| <1$}: Using the second order estimate in~\eqref{eq:extlogdiffsgeq1}, the difference of logarithms can be estimated by
	\begin{align*}
		&\left|
		\nabla G^\varepsilon * \log [g*G^\varepsilon](v) - \nabla G^\varepsilon * \log [g*G^\varepsilon](w)
		\right| 	\le \frac{4}{\varepsilon^2}|v-w|,
	\end{align*}
	giving the first estimate in the minimum.
	\\
	\underline{$|v-w| \ge 1$}: Bluntly apply the first order estimate in~\eqref{eq:extlogdiffsgeq1} onto each of the logarithms
	\[
	\|\nabla G^\varepsilon * \log [g*G^\varepsilon]\|_{L^\infty} \le \frac{1}{\varepsilon}.
	\]
	%
	
	The estimate for $U^\varepsilon$ follows by recalling $U^\varepsilon[g,f](v) = \int_{\Rd} K_g(v,w) \, \rd f(w)$ and Peetre's inequality from~\Cref{lem:peetre} for the case $\gamma\in(-2,0]$.
\end{proof}
The following estimate is adapted from~\cite[Equation (28)]{H09}.
\begin{lemma}[Pointwise difference in $K$]
	\label{lem:contourestimate}
	Fix $\varepsilon>0, \, g \in \mathscr{P}(\R^d)$, and $\gamma\in [-3,0]$. We have
	\[
	|K_g(v_1,w) - K_g(v_2,w)| \lesssim_{\varepsilon,\gamma}|v_1-v_2|\max\left(|v_1-w|^{2+\gamma},|v_2-w|^{2+\gamma}\right).
	\]
\end{lemma}
For completeness, we refer to~\Cref{sec:contourproof} for the proof of~\Cref{lem:contourestimate}.
\begin{proposition}[H\"older continuity of $U$]
	\label{lem:UHold}
	Fix $\gamma\in[-3,-2]$ and $g\in\mathscr{P}(\R^d)$. Then we have
	\[
	|U^\varepsilon[g](v_1) - U^\varepsilon[g](v_2)|\lesssim_\varepsilon |v_1 - v_2|^{3+\gamma}.
	\]
\end{proposition}
\begin{proof}
	We split the integration region into two cases
	\begin{align*}
		&\quad U^\varepsilon[g](v_1) - U^\varepsilon[g](v_2) = \int_{\R^d}(K_g(v_1,w) - K_g(v_2,w)) \,\rd g(w) \\
		&= \left(
		\int_{|v_1-v_2| \le \min (|v_1-w|, \, |v_2-w|)} + \int_{|v_1-v_2| > \min (|v_1-w|, \, |v_2-w|)}
		\right)(K_g(v_1,w) - K_g(v_2,w)) \,\rd g(w) \\
		&=: I_1 + I_2.
	\end{align*}
	We claim both $|I_1|, \, |I_2| \lesssim_\varepsilon |v_1-v_2|^{3+\gamma}$. Starting with $I_1$ where $|v_1-v_2| \le \min (|v_1-w|, \, |v_2-w|)$, we use~\Cref{lem:contourestimate} and the fact that $2+\gamma\le 0$ to deduce
	\begin{align*}
		|I_1| &\lesssim_\varepsilon |v_1-v_2| \int_{|v_1-v_2| \le \min (|v_1-w|, \, |v_2-w|)} \max (|v_1-w|^{2+\gamma}, \, |v_2-w|^{2+\gamma}) \, \rd g(w) \\
		&\le |v_1-v_2|^{3+\gamma}.
	\end{align*}
	Turning to $I_2$ the other integration region, assume without loss of generality that $|v_1 - v_2| > |v_1 - w|$. By the triangle inequality we also have
	\[
	|v_2-w| \le 2|v_1 - v_2|.
	\]
	Putting these two estimates together and using the bound $|K_g(v,w)|\lesssim_\varepsilon \min (|v-w|^{3+\gamma}, \, |v-w|^{2+\gamma})$ from~\Cref{prop:boundsKg}, we have
	\begin{align*}
		&\quad |I_2| \lesssim_\varepsilon
		\int_{|v_1-v_2| >  \min (|v_1-w|, \, |v_2-w|)} \min (|v_1-w|^{3+\gamma}, |v_1 - w|^{2+\gamma}) \,\rd g(w) + \dots \\
		&\quad + \int_{|v_1-v_2| >  \min (|v_1-w|, \, |v_2-w|)}\min (|v_2-w|^{3+\gamma}, |v_2 - w|^{2+\gamma}) \,\rd g(w) \\
		&\lesssim |v_1-v_2|^{3+\gamma}.
	\end{align*}
\end{proof}

We can improve~\Cref{lem:UHold} to Lipschitz continuity by taking advantage of extra regularity properties of $g$.
\begin{proposition}[Lipschitz continuity of $U$]
	\label{prop:ULipLp}
	Fix $g\in\mathscr{P}(\R^d), \, \varepsilon>0$, and $\gamma\in(-3,0]$. In the case $\gamma\in(-3,-2)$, assume further that $g$ satisfies~\ref{A:gamma}. Then we have
	\[
	|U^\varepsilon[g](v^1) - U^\varepsilon[g](v^2)| \le \Lambda_\gamma(g,v^1, v^2) |v^1 - v^2|,
	\]
	where
	\[
	\Lambda_\gamma = \left\{
	\begin{array}{cl}
		C_\varepsilon M_{2+\gamma}(g)\left(
		\japangle{v^1}^{2+\gamma}+ \japangle{v^2}^{2+\gamma}
		\right), 	&\gamma\in [-2,0] 	\\
		C_{\varepsilon, \, \gamma, \, p', \, d}(1+ \|g\|_{L^p}),     &\gamma\in(-3,-2)
	\end{array}
	\right.,
	\]
	and the constants $C>0$ only depend on the quantities in the subscript.
\end{proposition}
\begin{proof}
	The starting point is the application of~\Cref{lem:contourestimate} to first write
	\begin{align*}
		&\quad \left|U^\varepsilon[g](v^1) - U^\varepsilon[g](v^2)\right| = \left|\int_{\R^d} (K_g(v^1, w) - K_g(v^2, w)) \,\rd g(w)\right| \\
		&\lesssim_\varepsilon |v^1-v^2| \int_{\R^d} \max (|v^1-w|^{2+\gamma}, \, |v^2-w|^{2+\gamma}) \,\rd g(w) \\
		&\le |v^1-v^2| \underbrace{\int_{\R^d} (|v^1-w|^{2+\gamma} + |v^2-w|^{2+\gamma})\, \rd g(w)}_{=:I}.
	\end{align*}
	We estimate $I$ depending on the value of $\gamma$.
	\\
	\ul{The case $\gamma\in(-3,-2)$}: By splitting the integration region and using the fact that $g\in L^p$, we obtain
	\begin{align*}
		I &\lesssim 1 + \sup_{v \in \R^d}\int_{|v - w| \le 1}|v-w|^{2+\gamma}\, \rd g(w) \\
		&\le 1 + \sup_{v \in \R^d}\left(
		\int_{|v-w|\le 1}|v-w|^{(2+\gamma)p'}\,\rd w
		\right)^\frac{1}{p'}\|g\|_{L^p} \\
		&\lesssim_{\gamma, \, p, \, d}1 + \|g\|_{L^p}.
	\end{align*}
	\ul{The case $\gamma\in[-2,0]$}: The integrand is no longer singular so we use Peetre's inequality from~\Cref{lem:peetre}
	\[
	|v - w|^{2+\gamma} \le \japangle{v-w}^{2+\gamma} \lesssim_{\gamma}\japangle{v}^{2+\gamma}\japangle{w}^{2+\gamma}
	\]
	for $v = v^1, \, v^2$ into $I$ to get
	\begin{align*}
		I &\lesssim \int_{\R^d}\left(\japangle{v^1}^{2+\gamma} + \japangle{v^2}^{2+\gamma}\right)\japangle{w}^{2+\gamma} \, \rd g(w).
	\end{align*}
\end{proof}
\subsection{The velocity field as a function of measures}
The previous results established estimates for the pointwise variation of $K_g$ and $U^\varepsilon[g]$ given a fixed measure $g$. We now investigate the measure-wise variation of $K$ and $U^\varepsilon$ given a fixed point.
\begin{lemma}
	\label{lem:difflog}
	Fix $\varepsilon>0$ and let $\tau$ be an optimal transport map in $W_\infty$ between $f,g \in \mathscr{P}_c(\R^d)$ so that $g = \tau \# f$. Then, we have
	\[
	\left|\log [g*G^\varepsilon](v) - \log [f*G^\varepsilon](v)\right| \le \frac{1}{\varepsilon}W_\infty(g,f), \quad \forall v \in \R^d.
	\]
\end{lemma}
\begin{proof}
	We first note that
	\[
	g*G^\varepsilon(v) = \int_{\R^d} G^\varepsilon(v-w) \,\rd g(w) = \int_{\R^d}G^\varepsilon(v - \tau(w)) \,\rd f(w).
	\]
	Using the fundamental theorem of calculus, we express the difference of the logarithms as
	\begin{align*}
		\log [g*G^\varepsilon](v) &- \log [f*G^\varepsilon](v) \\ &= \int_0^1 \frac{\rd}{\rd t}\log \left[
		\int_{\R^d}G^\varepsilon(v - (t\tau(w) + (1-t)w)) \,\rd f(w)
		\right]\rd t \\
		&= -\int_0^1 \frac{\int_{\R^d}(\tau(w) - w)\cdot \nabla G^\varepsilon(v - (t\tau(w) + (1-t)w)) \,\rd f(w)}{\int_{\R^d}G^\varepsilon(v - (t\tau(w) + (1-t)w) \,\rd f(w)}\, \rd t.
	\end{align*}
	By definition, we have $|\tau(w) - w| \le W_\infty(g,f)$ and moreover recalling~\eqref{eq:diffGseps}, we see that
	\[
	|\nabla G^\varepsilon| \le \frac{1}{\varepsilon}G^\varepsilon.
	\]
	Applying these two estimates into the previous computations, we have
	\[
	\left|
	\log [g*G^\varepsilon](v) - \log [f*G^\varepsilon](v)
	\right| \le \frac{W_\infty(g,f)}{\varepsilon}.
	\]
\end{proof}
The following technical estimates hinge on~\Cref{lem:difflog}. 
\begin{lemma}[Measure-wise difference in $K$]
	\label{lem:diffKg}
	Fix $g^i\in\mathscr{P}_c(\R^d)$ for $i=1, \, 2$ and $\gamma\in[-3,0]$. Then for every $v,\, w\in \R^d$, we have the estimate
	\[
	|K_{g^1}(v,w) - K_{g^2}(v,w)| \lesssim \min \left(
	\frac{1}{\varepsilon^3}|v-w|^{3+\gamma}, \, \frac{2}{\varepsilon^2}|v-w|^{2+\gamma}
	\right) W_\infty(g^1, \, g^2).
	\]
\end{lemma}
\begin{proof}
	Here, we need to estimate
	\begin{align*}
		&\quad K_{g^1}(v,w) - K_{g^2}(v,w) \\ &= -|v-w|^{2+\gamma}\Pi[v-w] \left(
		\nabla G^\varepsilon * \left\{
		\log [g^1*G^\varepsilon] - \log [g^2*G^\varepsilon]
		\right\}(v) -\nabla G^\varepsilon * \left\{
		\log [g^1*G^\varepsilon] - \log [g^2*G^\varepsilon]
		\right\}(w)
		\right).
	\end{align*}
	By the fundamental theorem of calculus, we have an estimate for this difference
	\begin{align*}
		&\quad  |v-w|^{2+\gamma}\left|\int_0^1 
		\frac{\rd}{\rd t}\nabla G^\varepsilon*\{
		\log [g^1*G^\varepsilon] - \log [g^2*G^\varepsilon]
		\}(tv + (1-t)w)\rd t
		\right| \\
		&\le |v-w|^{3+\gamma}\left\|
		\nabla^2 G^\varepsilon * \left\{\log [g^1*G^\varepsilon] - \log [g^2*G^\varepsilon]\right\}
		\right\|_{L^\infty}.
	\end{align*}
	Using~\Cref{lem:difflog} and the comparison $|\nabla^2 G^\varepsilon| \lesssim \frac{1}{\varepsilon^2} G^\varepsilon$ (an extension of~\eqref{eq:diffGseps}), we apply Young's convolution inequality to deduce
	\[
	|K_{g^1}(v,w) - K_{g^2}(v,w)| \lesssim \frac{1}{\varepsilon^3}|v-w|^{3+\gamma}W_\infty(g^1,g^2).
	\]
	On the other hand, without estimating second order derivatives, we can bluntly prove
	\begin{align*}
		|K_{g^1}(v,w) - K_{g^2}(v,w)| &\le |v-w|^{2+\gamma}\left(\left|
		\nabla G^\varepsilon * \left\{
		\log [g^1*G^\varepsilon] - \log [g^2*G^\varepsilon]
		\right\}(v)\right|\right. \\
		&\quad +\left.\left|\nabla G^\varepsilon * \left\{
		\log [g^1*G^\varepsilon] - \log [g^2*G^\varepsilon]
		\right\}(w)
		\right|\right) \\
		&\le 2|v-w|^{2+\gamma}\| \nabla G^\varepsilon * \left\{
		\log [g^1*G^\varepsilon] - \log [g^2*G^\varepsilon]
		\right\} \|_{L^\infty}.
	\end{align*}
	Recalling~\eqref{eq:diffGseps} and~\Cref{lem:difflog} which say
	\[
	|\nabla G^\varepsilon| \le \frac{1}{\varepsilon}G^\varepsilon, \quad \left|
	\log [g^1*G^\varepsilon] - \log [g^2*G^\varepsilon]
	\right| \le \frac{1}{\varepsilon}W_\infty(g^1,g^2),
	\]
	we use Young's convolution inequality again to get
	\[
	|K_{g^1}(v,w) - K_{g^2}(v,w)| \le \frac{2}{\varepsilon^2}|v-w|^{2+\gamma}W_\infty(g^1,g^2).
	\]
\end{proof}
Under minimal assumptions on the probability measures, we can obtain a H\"older estimate with respect to the $W_\infty$ metric.
\begin{lemma}
	\label{lem:gfvel}
	Fix $g^i, \, f^i \in \mathscr{P}(\R^d)$ for $i=1, \, 2$ and $\gamma\in(-3,-2]$. Then, we have the estimate
	\[
	|U^\varepsilon[g^1,f^1](v) - U^\varepsilon[g^2,f^2](v)| \lesssim W_\infty(g^1,g^2) + W_\infty(f^1,f^2)^{3+\gamma}.
	\]
\end{lemma}
\begin{proof}
	Starting from the definition, we have
	\begin{align*}
		&\quad U^\varepsilon[g^1,\, f^1](v) - U^\varepsilon[g^2,\, f^2](v)   \\ &=  \int_{\R^d}K_{g^1}(v,w) \, \rd f^1(w) - \int_{\R^d}K_{g^2}(v,w)\, \rd f^2(w)     \\
		&= \int_{\R^d}(K_{g^1}(v,w) - K_{g^2}(v,w))\, \rd f^1(w) + \int_{\R^d}K_{g^2}(v,w) \, \rd (f^1 - f^2)(w) \\
		&=: I_1 + I_2.
	\end{align*}
	We claim the following estimates
	\[
	|I_1| \lesssim_\varepsilon W_\infty(g^1,g^2), \quad |I_2| \lesssim_\varepsilon W_\infty(f^1,f^2)^{3+\gamma}.
	\]
	The term $I_1$ is almost completely treated by~\Cref{lem:diffKg}. We can further estimate the minimum by
	\begin{align*}
		&\quad |I_1| \lesssim_\varepsilon W_\infty(g^1,g^2) \int_{\R^d} \min (|v-w|^{3+\gamma},|v-w|^{2+\gamma}) \, \rd f^1(w) \\
		&\lesssim W_\infty(g^1,g^2) \int_{\R^d}\japangle{v-w}^{2+\gamma} \,\rd f^1(w).
	\end{align*}
	When $\gamma\in(-3,-2]$, simply estimate $\japangle{v-w}^{2+\gamma} \le 1$. This takes care of $I_1$ so we focus on $I_2$ for the rest of this proof.
	
	Firstly, take $\tau$ an optimal transport map in $W_\infty$ between $f^1$ and $f^2$ i.e.
	\[
	W_\infty(f^1,f^2) = \mathrm{esssup}_{w\in\R^d}|\tau(w) - w|.
	\]
	Moreover, the following identity holds in $\mathscr{P}(\R^d); \, f^2 = \tau\# f^1.$ This allows us to rewrite the difference $I_2$ as
	\begin{equation}
		\label{eq:Itwotau}
		I_2 = \int_{\R^d}\left[
		K_{g^2}(v,w) - K_{g^2}(v,\tau(w))
		\right] \,\rd f^1(w).
	\end{equation}
	We split the integration region in~\eqref{eq:Itwotau} into $\mathcal{A} = \{w \in \R^d \, | \, |v-w| < 2 W_\infty(f^1,f^2)\}$ and its complement $\R^d \setminus \mathcal{A}$. In the set $\mathcal{A}$, we begin with the blunt $L^\infty$ bound on $K$ from~\Cref{prop:boundsKg} which implies
	\begin{align*}
		&\quad |K_{g^2}(v,w) - K_{g^2}(v,\tau(w)) | \\
		&\lesssim_{\varepsilon} \min (|v-w|^{3+\gamma}, \, |v-w|^{2+\gamma}) + \min (|v-\tau(w)|^{3+\gamma}, \, |v-\tau(w)|^{2+\gamma}) \\
		&\lesssim W_\infty(f^1,f^2)^{3+\gamma} + |v-\tau(w)|^{3+\gamma}.
	\end{align*}
	For the second term, we simply use the triangle inequality
	\[
	|v-\tau(w)| \le |v-w| + |\tau(w) - w| \le 3W_\infty(f^1,f^2).
	\]
	This gives
	\begin{equation}
		\label{eq:diffKItwo}
		|K_{g^2}(v,w) - K_{g^2}(v,\tau(w)) | \lesssim_{\varepsilon} W_\infty(f^1,f^2)^{3+\gamma}
	\end{equation}
	which is independent of $w$ so we have
	\[
	\int_{\mathcal{A}}|K_{g^2}(v,w) - K_{g^2}(v,\tau(w)) | \,\rd f^1(w) \lesssim_{\varepsilon} W_\infty(f^1,f^2)^{3+\gamma}.
	\]
	Turning to the complement region $\R^d \setminus \mathcal{A}$ given as $\{ w \in \R^d \, | \, |v-w| \ge 2W_\infty(f^1,f^2)\},$ we use~\Cref{lem:contourestimate} to obtain
	\[
	|K_{g^2}(v,w) - K_{g^2}(v,\tau(w)) | \lesssim_{\varepsilon} |\tau(w) - w| \max (|v-w|^{2+\gamma}, \, |v-\tau(w)|^{2+\gamma}).
	\]
	Recalling that $2+\gamma \le 0$, the reverse triangle inequality yields
	\[
	|v-\tau(w)| \ge |v-w| - |\tau(w) - w| \ge W_\infty(f^1,f^2),
	\]
	because $|v-w|\ge 2W_\infty(f^1,f^2)$.
	
	Therefore, from the previous estimate, we obtain
	\[
	|K_{g^2}(v,w) - K_{g^2}(v,\tau(w))| \lesssim_\varepsilon |\tau(w) - w| W_\infty(f^1,f^2)^{2+\gamma} \le W_\infty(f^1, f^2)^{3+\gamma}.
	\]
	This is exactly the same as~\eqref{eq:diffKItwo} for $\mathcal{A}$. Integrating both inequalities against $f^1$ yields
	\begin{align*}
		|I_2| &\le \left(
		\int_{\mathcal{A}} + \int_{\R^d \setminus \mathcal{A}}
		\right)|K_{g^2}(v,w) - K_{g^2}(v,\tau(w)) | \, \rd f^1(w) \\
		&\lesssim_\varepsilon W_\infty(f^1,f^2)^{3+\gamma}.
	\end{align*}
\end{proof}
If we impose more assumptions on the probability measures, in particular~\ref{A:gamma}, we can derive linear stability with respect to the $W_\infty$ metric.
\begin{proposition}[Linear stability]
	\label{prop:linest}
	Fix $g^i, \, f^i \in \mathscr{P}_c(\R^d)$ for $i=1, \, 2$ and $\gamma\in(-3,0]$.
	For $\gamma\in(-3,-2)$, assume further that $f^i$ satisfies~\ref{A:gamma}. Then we have the estimate
	\begin{align*}
		&\quad |U^\varepsilon[g^1, \, f^1](v) - U^\varepsilon[g^2, \, f^2](v)| \lesssim_\varepsilon 	\\ 
		&\left\{
		\begin{array}{cl}
			\japangle{v}^{2+\gamma}\left[M_{2+\gamma}(f^1)W_\infty(g^1, \, g^2) + \left\{M_{2+\gamma}(f^1) + M_{2+\gamma}(f^2)\right\}W_\infty(f^1, \, f^2)\right], 	&\gamma\in[-2,0] 	\\
			W_\infty(g^1, \, g^2) + (1 + \|f^1\|_{L^p} + \|f^2\|_{L^p})W_\infty(f^1, \, f^2), 	&\gamma\in(-3,-2)
		\end{array}
		\right..
	\end{align*}
\end{proposition}
\begin{proof}
	Our starting point repeats the proof of~\Cref{lem:gfvel} above. Using the same notation from there, we split
	\begin{align*}
		&\quad U^\varepsilon[g^1, \, f^1](v) - U^\varepsilon[g^2, \, f^2](v) \\
		&= \int_{\R^d}(K_{g^1}(v,w) - K_{g^2}(v,w))\, \rd f^1(w)  + \int_{\R^d}K_{g^2}(v,w) \, \rd (f^1 - f^2)(w) \\
		&=: I_1 + I_2.
	\end{align*}
	We inherit the estimate for $I_1$ from the proof of~\Cref{lem:gfvel} which reads, using Peetre's inequality in~\Cref{lem:peetre} for $\gamma\in[-2,0]$,
	\[
	|I_1| \lesssim_\varepsilon W_\infty(g^1, \, g^2) \times \left\{
	\begin{array}{cl}
		1,    &\gamma\in (-3,-2)  \\
		M_{2+\gamma}(f^1) \japangle{v}^{2+\gamma},    &\gamma\in[-2,0] 
	\end{array}
	\right..
	\]
	We focus entirely on $I_2$; in the case $\gamma\in(-3,-2)$, we claim that
	\[
	|I_2| \lesssim_\varepsilon(1 + \|f^1\|_{L^p} + \|f^2\|_{L^p}) W_\infty(f^1, \, f^2).
	\]
	In the case $\gamma\in[-2,0]$, we claim that
	\[
	|I_2| \lesssim_\varepsilon \japangle{v}^{2+\gamma} (M_{2+\gamma}(f^1) + M_{2+\gamma}(f^2))W_\infty(f^1, \, f^2).
	\]
	In both cases, we rewrite $I_2$ in the following way; take $\tau$ an optimal transport map between $f^1$ and $f^2$ in $W_\infty$ so that we have
	\[
	I_2 = \int_{\R^d} (K_{g^2}(v,w) - K_{g^2}(v,\tau(w))\, \rd f^1(w), \quad f^2 = \tau \# f^1.
	\]
	Applying~\Cref{lem:contourestimate} and recalling the (anti-)symmetry of $K_{g^2}(v,w) = -K_{g^2}(w,v)$, we have
	\begin{align*}
		|I_2| &\lesssim_\varepsilon \int_{\R^d} |w - \tau(w)| \max \left(
		|v-w|^{2+\gamma}, \, |v-\tau(w)|^{2+\gamma}
		\right)\, \rd f^1(w) \\
		&\le W_\infty(f^1, \, f^2) \left(
		\int_{\R^d} (|v-w|^{2+\gamma} + |v-\tau(w)|^{2+\gamma})\, \rd f^1(w)
		\right).
	\end{align*}
	We split the sum and reformulate the second term in terms of $f^2$ using $\tau$ to obtain
	\begin{equation}
		\label{eq:splitlin}
		|I_2| \lesssim_\varepsilon W_\infty(f^1, \, f^2) \left(
		\int_{\R^d} |v-w|^{2+\gamma} \,\rd f^1(w) + \int_{\R^d} |v-w|^{2+\gamma} \,\rd f^2(w)
		\right).
	\end{equation}
	\ul{The case $\gamma\in(-3,-2)$}: By partitioning $\R^d$ into $\{w\in \R^d \, | \, |v-w|\le 1\}$ and its complement, a standard application of H\"older's inequality gives
	\[
	\int_{\R^d}|v-w|^{2+\gamma}\,\rd f^1(w) \le 1 + \left(
	\int_{|v-w| \le 1}|v-w|^{(2+\gamma)p'}
	\right)^\frac{1}{p'} \| f^1\|_{L^p}
	\]
	and similarly for $f^2$. Using the assumption $(2+\gamma)p' > -d$, inequality~\eqref{eq:splitlin} is further refined to
	\[
	|I_2| \lesssim_{\varepsilon, \, d, \, \gamma, \, p} (1 + \|f^1\|_{L^p} + \|f^2\|_{L^p}) W_\infty(f^1, \, f^2).
	\]
	\ul{The case $\gamma\in[-2,0]$}: Since $2+\gamma\ge 0$, we use Peetre's inequality~\Cref{lem:peetre} to estimate
	\[
	|v-w|^{2+\gamma} \le \japangle{v-w}^{2+\gamma} \lesssim \japangle{v}^{2+\gamma} \japangle{w}^{2+\gamma}.
	\]
	Inserting this into~\eqref{eq:splitlin}, we get
	\[
	|I_2| \lesssim_\varepsilon \japangle{v}^{2+\gamma} W_\infty(f^1, \, f^2) \left(\int_{\R^d}\japangle{w}^{2+\gamma} \,\rd (f^1 + f^2)(w)\right).
	\]
\end{proof}

\section{The continuum model}
\label{sec:existcts}
This section is devoted to the proof of~\Cref{thm:existmf}; the well-posedness of~\eqref{eq:epslan}. To fix notation, we seek solutions in the following spaces
\[
X_\gamma = X_\gamma(T):= \left\{
\begin{array}{ll}
	C([0,T]; \mathscr{P}_c(\R^d)), 	&\gamma\in[-2,0] 	\\
	C([0,T]; \mathscr{P}_c(\R^d)) \cap L^\infty(0,T; L^p(\R^d)), 	&\gamma\in(-3,-2)
\end{array}
\right..
\]
For $\gamma\in(-3,-2)$, the exponent $p$ corresponds to that of~\ref{A:gamma}.
In particular, we endow $X_\gamma$ with the metric
\[
d_\infty(f^1, \, f^2) := \sup_{t \in [0,T]}W_\infty(f^1(t),f^2(t)), \quad \forall f^1, \, f^2 \in X_\gamma.
\]
Given $g \in X_\gamma$, we first want to find $f \in X_\gamma$ solving
\begin{equation}
	\tag{P}
	\label{eq:contractmap}
	\partial_t f + \nabla\cdot(fU^\varepsilon[g])=0, \quad f(t=0) = f^0. 
\end{equation}
Well-posedness of~\eqref{eq:epslan} then comes from ensuring the map $g\mapsto f$ just described has a unique fixed point in a closed subspace of $X_\gamma$. 

Given a curve $g\in X_\gamma$, we denote by $\Phi_g$ the characteristic flow corresponding to~\eqref{eq:contractmap} satisfying
\begin{equation}
	\label{eq:contODE}
	\frac{\rd}{\rd t}\Phi_g(t,v) = U^\varepsilon[g](\Phi_g(t,v)), \quad \Phi_g(0,v) = v \in \R^d.
\end{equation}

\begin{proposition}
	\label{prop:Pwellposed}
	Fix $\varepsilon>0, \gamma\in(-3,0], \, g\in X_\gamma$, and initial condition $f^0$ satisfying~\ref{A:cpctsupp} and~\ref{A:gamma}. Then, $f(t) = \Phi_g(t,\cdot)\#f^0$ is the unique weak solution in $C([0,T];\mathscr{P}(\R^d))$ to~\eqref{eq:contractmap}.
\end{proposition}
\begin{proof}
	The (local) Lipschitz continuity of $U^\varepsilon[g]$ is provided by~\Cref{prop:ULipLp} so the characteristic system~\eqref{eq:contODE} has a unique solution, $\Phi_g$, up to the flow map's maximal time of existence $T^*>0$. By~\cite[Theorem 2.3.5]{G03}, $f(t) = \Phi_g(t,\cdot)\# f^0$ is the unique weak solution in $C([0,T^*];\mathscr{P}(\R^d))$. For $\gamma\in(-3,-2)$, \Cref{prop:ULipLp}, implies global Lipschitz regularity of $U^\varepsilon$ hence $\Phi_g$ is globally defined and we can directly take $T^* = T$. For $\gamma\in[-2,0]$, \Cref{lem:growthPhi} excludes blow up of $\Phi_g$ so  we can take $T^* = T$ here too.
\end{proof}

\begin{lemma}
	\label{lem:growthPhi}
	
	For $\varepsilon>0$ and $g\in X_\gamma$, let $\Phi_g$ be the flow map of~\eqref{eq:contODE} with maximal time of existence $T^*>0$. Then, we have the estimates
	\[
	\japangle{\Phi_g}\le \left\{
	\begin{array}{cl}
		\japangle{v}\exp\left\{
		C_\varepsilon \left[\sup_{s\in[0,T]}M_2(g(s))\right]t
		\right\}, 	&\gamma\in(-2,0] \\
		\japangle{v} + C_\varepsilon t, 	&\gamma\in(-3,-2]
	\end{array}
	\right., \quad \forall t\in[0,T^*].
	\]
	Here, $C_\varepsilon>0$ is a constant depending only on $\varepsilon>0$. In particular, $\Phi_g$ extends to a global solution of~\eqref{eq:contODE} on $[0,T]$. 
\end{lemma}
\begin{proof}
	We begin, for $\gamma\in(-2,0]$, by differentiating $\frac{1}{2}|\Phi_g(t,v)|^2$ with respect to $t \in (0,T^*)$. Expanding the definition of $U^\varepsilon$, we obtain
	\begin{align}
		\label{eq:secondmomODE}
		\begin{split}
			\quad &\frac{\rd}{\rd t}\frac{1}{2}|\Phi_g(t,v)|^2 = \Phi_g(t,v) \cdot U^\varepsilon[g](\Phi_g(t,v)) \\ 
			&= \Phi_g(t,v)\cdot \int_{\R^d}|\underbrace{\Phi_g(t,v)-w|^{2+\gamma} \Pi[\Phi_g(t,v) - w] B^\varepsilon (\Phi_g(t,v),w)}_{=:I}\, \rd g(w),
		\end{split}
	\end{align}
	where we have abbreviated 
	\[
	B^\varepsilon(v,w) := \nabla G^\varepsilon * \log [g*G^\varepsilon](v) - \nabla G^\varepsilon * \log [g*G^\varepsilon](w).
	\]
	Notice that $|B^\varepsilon|\lesssim_\varepsilon 1$ by~\eqref{eq:extlogdiffsgeq1}. Our goal is to show $\int I \, \rd g(w) \lesssim |\Phi_g|$ and then apply Gr\"onwall's inequality. First, we split the integral into regions where $|\Phi_g - w| \le 1$ and $|\Phi_g - w|> 1$. As $2+\gamma > 0$, the former piece can be easily estimated as follows
	\begin{align}
		\label{eq:secondmomsplit}
		\begin{split}
			\Phi_g \cdot \int_{\R^d} I \, \rd g(w) &= \Phi_g \cdot \left(
			\int_{|\Phi_g - w|\le 1} I \, \rd g(w) + \int_{|\Phi_g - w|> 1} I \, \rd g(w)
			\right)  \\
			&\lesssim_\varepsilon|\Phi_g| + \underbrace{\int_{|\Phi_g - w|> 1} |\Phi_g - w|^{2+\gamma}|\Pi[\Phi_g -w]\Phi_g| \, \rd g(w)}_{=:II}.
		\end{split}
	\end{align}
	Turning to $II$, we use the fact that $\Pi[v-w]v = \Pi[v-w]w$ and $|\Phi_g-w|^2\le 2|\Phi_g|^2+2|w|^2$ to estimate
	\begin{align*}
		II &\lesssim |\Phi_g|^2 \int_{|\Phi_g - w|> 1} |\Phi_g - w|^{\gamma}|\Pi[\Phi_g -w]w|\, \rd g(w) \\
		&\quad + \int_{|\Phi_g - w|> 1} |w|^2|\Phi_g - w|^{\gamma}|\Pi[\Phi_g -w]\Phi_g| \, \rd g(w).
	\end{align*}
	Finally, since $|\Phi_g - w|>1$ and $\gamma\le 0$, we can bluntly estimate the remaining contributions by
	\[
	II \le M_1(g)|\Phi_g|^2 + M_2(g) |\Phi_g|.
	\]
	Putting this together with~\eqref{eq:secondmomsplit} and~\eqref{eq:secondmomODE}, we have
	\[
	\frac{\rd}{\rd t}\frac{1}{2}\japangle{\Phi_g(t,v)}^2 \lesssim_\varepsilon M_2(g) \japangle{\Phi_g(t,v)}^2.
	\]
	Gr\"onwall's inequality gives the inequality for $\gamma\in(-2,0]$.
	
	Turning to the case $\gamma\in(-3,-2]$, we write the integral form of~\eqref{eq:contODE}
	\begin{align*}
		|\Phi_g(t,v) - v| = \left|
		\int_0^t U^\varepsilon[g](\Phi_g(s,v)) \,\rd s
		\right| \lesssim_\varepsilon t.
	\end{align*}
	The final estimate comes from applying~\Cref{prop:boundsKg}.
\end{proof}

\begin{proposition}
	\label{prop:Xcomplete}
	The space $(X_\gamma, \, d_\infty)$ for $\gamma\in(-3,0]$ is a complete metric space.
\end{proposition}
\begin{proof}
	This can be proven from the fact that $(\mathscr{P}_c(\R^d), W_\infty)$ is complete  and metrizes weak convergence~\cite{RR98}. Moreover, the $L^p$ norm is lower semi-continuous with respect to this topology.
\end{proof}


\subsection{Moment and $L^p$ propagation}
\label{sec:momprop}
In this section, we derive the moment and $L^p$ propagation estimates we will need for the fixed point argument to prove~\Cref{thm:existmf}.
\begin{proposition}
	\label{cor:energyprop}
	Fix $\varepsilon>0, \gamma\in(-3,0]$, and initial condition $f^0$ satisfying~\ref{A:cpctsupp} and~\ref{A:gamma} and $g\in X_\gamma$. The unique weak solution $f(t) = \Phi_g(t,\cdot)\#f^0$ to~\eqref{eq:contractmap} belongs to $C([0,T];\mathscr{P}_c(\R^d))$ with second moment growth estimate
	\[
	M_2(f(t)) \le M_2(f^0) \exp\left\{
	C_\varepsilon \left(\sup_{s\in[0,T]}M_2(g(s))\right)t
	\right\}, \quad \forall t \in [0,T].
	\]
	If, moreover, $f=g$ (i.e. $f$ solves~\eqref{eq:epslan}), then $M_2(f(t)) = M_2(f^0)$ for all $t\in[0,T]$.
\end{proposition}
\begin{remark}
	The same propagation result applies for higher order moments. The constant $C_\varepsilon$ grows linearly with the order of the moment.
\end{remark}
The following $L^p$ estimate can be derived directly from standard facts about solutions to the continuity equation which can be found, for example, in~\cite{G03,BLR11}. For completeness, we prove~\Cref{lem:fLpest} in~\Cref{sec:fLpest}.
\begin{lemma}
	\label{lem:fLpest}
	For fixed $\varepsilon>0, \, \gamma\in (-3,-2)$ and $g\in X_\gamma$, we have the following $L^p$ estimate for $f(t) = \Phi_g(t,\cdot)\# f^0$ where $f^0 \in L^p$ and $\Phi_g$ are as in~\eqref{eq:contractmap} and~\eqref{eq:contODE}, respectively.
	\[
	\|f(t)\|_{L^p}\le \|f^0\|_{L^p}\exp\left\{
	C_{\varepsilon,\, \gamma, \, d}\left(1+\mathrm{esssup}_{s\in [0,T]}\|g(s)\|_{L^p}\right)t
	\right\}, \quad \forall t \in [0,T].
	\]
	In particular, $f \in X_\gamma$.
\end{lemma}
\begin{proof}[Proof of~\Cref{cor:energyprop}]
	We begin by writing the weak formulation of~\eqref{eq:contractmap} against test functions $\phi \in C_c^\infty(\R^d)$,
	\begin{align}
		\label{eq:wkcontract}
		\begin{split}
			&\quad \frac{\rd}{\rd t}\int_{\R^d} \phi(v) \, \rd f_t(v) = \int_{\R^d} \nabla \phi(v) \cdot U^\varepsilon[g(t)](v) \, \rd f_t(v) = \\
			&\iint_{\R^{2d}} |v-w|^{2+\gamma}\nabla \phi(v) \cdot \Pi[v-w]B^\varepsilon \, \rd f(v) \, \rd g(w).
		\end{split}
	\end{align}
	We will derive the desired estimate by making use of $\japangle{v}^2\phi_R$ as a test function in~\eqref{eq:wkcontract}, where 
	\[
	\phi_R(v) = \phi\left(\frac{v}{R}\right), \quad \phi(v) = \Phi(|v|) = \left\{
	\begin{array}{cc}
		1 	&|v|\le 1 	\\
		0 	&|v| > 2
	\end{array}
	\right., \quad \Phi \ge 0, \quad \Phi \in C_c^\infty(\R).
	\]
	Firstly, notice that $\nabla (\japangle{v}^2 \phi_R(v))$ is supported in $|v| \le 2R$ and takes the form
	\begin{equation}
		\label{eq:nablapoly}
		\nabla (\japangle{v}^2 \phi_R(v)) = \underbrace{\left(2\phi_R(v) 
			+ \frac{1}{R}\japangle{v}^2 \frac{\Phi'(|v|/R)}{|v|}
			\right)}_{=:P(v)}v.
	\end{equation}
	In particular, since $\frac{1}{R} \lesssim \frac{1}{|v|} \lesssim \frac{1}{\japangle{v}}$ for large $R\gg 1$, the bound for $\Phi'$ gives
	\begin{equation}
		\label{eq:nablapolyest}
		\nabla (\japangle{v}^2 \phi_R(v)) = P(v)v, \quad\text{with } |P(v)| \lesssim 1.
	\end{equation}
	We start with the easier case of $\gamma\in(-3,-2)$. By interpolating the estimates in~\eqref{eq:extlogdiffsgeq1}, the function $\nabla \log [g * G^\varepsilon]$ is H\"older continuous so we can deduce
	\[
	|v-w|^{2+\gamma}|B^\varepsilon| \lesssim_\varepsilon 1, \quad B^\varepsilon = \nabla G^\varepsilon * \log [g*G^\varepsilon](v) - \nabla G^\varepsilon * \log [g*G^\varepsilon](w).
	\]
	This greatly simplifies the double integral to the following
	\begin{align*}
		&\quad \left|\iint_{\R^{2d}}|v-w|^{2+\gamma}\nabla (\japangle{v}^2 \phi_R(v))\cdot \Pi[v-w]B^\varepsilon \, \rd f(v)\,\rd g(w)\right| 	\\
		&\le C_\varepsilon\iint_{\R^{2d}} \japangle{v} \, \rd f(v) \, \rd g(w) \le C_\varepsilon  M_{1}(f).
	\end{align*}
	This establishes (a stronger version of) the result for $\gamma\in(-3,-2)$. Turning to the case $\gamma\in[-2,0]$, we split the inner integral in $v$ of~\eqref{eq:wkcontract} into regions where $|v-w|\le 1$ and $|v-w|\ge 1$ obtaining a first reduction using~\eqref{eq:extlogdiffsgeq1} and~\eqref{eq:nablapolyest}
	\begin{align}
		\label{eq:noconvfirst}
		\begin{split} 
			&\quad \int_{w\in\R^d}\left(
			\int_{\{v \, : \,|v-w|\le 1\}} + \int_{\{v \, : \, |v-w|\ge 1\}}
			\right)|v-w|^{2+\gamma}P(v)v\cdot \Pi[v-w]B^\varepsilon \, \rd f(v) \, \rd g(w) \\
			&\lesssim_\varepsilon M_0(g) M_{1}(f) + \underbrace{\int_{w\in\R^d} \int_{\{v\, : \,|v-w|\ge 1\}}|v-w|^{2+\gamma}P(v)v\cdot \Pi[v-w]B^\varepsilon \, \rd f(v) \, \rd g(w)}_{=:I}.
		\end{split}
	\end{align}
	For the term $I$, we use the identity $\Pi[v-w]v = \Pi[v-w]w$,  $|B^\varepsilon|\lesssim_\varepsilon 1$, \eqref{eq:nablapolyest}, and Young's inequality (c.f. the proof of~\Cref{lem:growthPhi}) which give
	\begin{align*}
		|I| &\lesssim_\varepsilon  \int_{w\in\R^d} \int_{\{ v \, : \, |v-w|\ge 1\}} (|v|^2 + |w|^2)|v-w|^\gamma |\Pi[v-w]v|\, \rd f(v) \, \rd g(w) \\
		&\le  \int_{w\in\R^d} \int_{\{v \, : \, |v-w| \ge 1\}}|v|^2|\Pi[v-w]w| \, \rd f(v) \, \rd g(w) \\
		&\quad + \int_{w\in\R^d} |w|^2 \int_{\{ v \, : \,|v-w|\ge 1\}}|\Pi[v-w]v|\, \rd f(v) \, \rd g(w) \\
		&\le M_1(g) M_2(f) +M_2(g) M_{1}(f).
	\end{align*}
	Collecting this estimate with~\eqref{eq:noconvfirst}, we have
	\begin{align*}
		&\quad \left|\iint_{\R^{2d}}|v-w|^{2+\gamma}\nabla (\japangle{v}^2 \phi_R(v))\cdot \Pi[v-w]B^\varepsilon \, \rd f(v) \, \rd g(w)\right|  \\
		&\lesssim_\varepsilon (M_0(g)M_{1}(f) + M_1(g)M_2(f) + M_2(g) M_{1}(f)) \le M_2(g) M_2(f).
	\end{align*}
	Inserting this estimate into~\eqref{eq:wkcontract} and integrating in time, we get
	\begin{align*}
		\int_{\Rd} \japangle{v}^2 \phi_R(v) \, \rd f_t(v)\le \int_{\Rd} \japangle{v}^2 \phi_R(v) \, \rd f^0(v) + C_\varepsilon\int_0^t M_2(g(s))M_2(f(s))\rd s.
	\end{align*}
	By Monotone Convergence, passing to the limit $R\to \infty$ with Gr\"onwall's inequality gives the stated a priori estimate on the growth of the second moment of $f$.
	
	Concerning the statement that $f\in C([0,T];\mathscr{P}_c(\R^d))$, notice that $\Phi_g$ is bounded in $[0,T]$ when $g\in X_\gamma$ according to~\Cref{lem:growthPhi}. Moreover, since $f^0$ has compact support, $f$ has compact support as a push-forward of $f^0$ through a bounded flow map.
	
	Finally, in the case $f=g$, take $\phi(v) = \japangle{v}^2 = 1+|v|^2$ as a test function (justified by the previous estimates) so that the right-hand side of~\eqref{eq:wkcontract} reads
	\begin{align*}
		&\quad \frac{\rd}{\rd t} M_2(f(t)) = \frac{\rd}{\rd t}\int_{\Rd} (1+|v|^2) \, \rd f_t(v) =2\iint_{\R^d}|v-w|^{2+\gamma} v \cdot \Pi[v-w]B^\varepsilon \, \rd f_t(v) \, \rd f_t(w) 	\\	= &\iint_{\R^d} |v-w|^{2+\gamma}(v-w)\cdot \Pi[v-w]B^\varepsilon \, \rd f_t(v) \, \rd f_t(w) = 0.
	\end{align*}
	In plain words, $M_2(f(t))$ is constant.
\end{proof}

\subsection{Well-posedness of~\eqref{eq:epslan}}
\label{sec:wpmf}
In this section, we prove~\Cref{thm:existmf} by applying the contraction mapping theorem to the solution map of~\eqref{eq:contractmap} denoted by $S_\gamma : g\in X_\gamma \mapsto f(t) = \Phi_g(t,\cdot)\#f^0 \in X_\gamma$ (c.f.~\Cref{cor:energyprop} and~\Cref{lem:fLpest}).
We will leverage the propagation estimates from~\Cref{sec:momprop} to the following closed subspace of $X_\gamma$. For $T>0$, define
\[
M_\gamma = M_\gamma(T) := \left\{
\begin{array}{ll}
	\{f \in X_\gamma \,| \,  \sup_{t\in[0,T]}M_2(f(t)) \le 2M_2(f^0)\}, 	&\gamma\in[-2,0] 	\\
	\{f \in X_\gamma \,| \, \mathrm{esssup}_{t\in[0,T]}\|f(t)\|_{L^p} \le 3\|f^0\|_{L^p}\}, 	&\gamma\in(-3,-2)
\end{array}
\right..
\]
\begin{remark}
	\label{lem:expbdd}
	Fix $b>0$ and $k> a > 0$. For $T>0$, define the function
	\[
	F_T : [0,k]\to [0,+\infty), \quad F_T(x) = ae^{bTx}.
	\]
	Then, for every $T\le T_C := \frac{1}{bk}\log \frac{k}{a}$, it holds $F_T \le k$. In other words, $F_T :[0,k]\to [0,k].$
\end{remark}
Motivated by~\Cref{lem:expbdd}, we define the time horizon
\[
T_C := \left\{
\begin{array}{cl}
	\frac{\log 2}{2C_\varepsilon M_2(f^0)}, 	&\gamma\in[-2,0] 	\\
	\min \left(
	\frac{\log 2}{C_{\varepsilon, \gamma, d}}, \, \frac{\log \frac{3}{2}}{6 C_{\varepsilon, \gamma, d}\|f^0\|_{L^p}}
	\right), 	&\gamma\in(-3,-2)
\end{array}
\right.,
\]
where $C_\varepsilon, \, C_{\varepsilon, \gamma,d}>0$ are the constants appearing in the exponential in~\Cref{cor:energyprop} and~\Cref{lem:fLpest}, respectively. The plan of the following proof is to show that $S_\gamma$ is a contraction from $M_\gamma$ to itself.
\begin{proof}[Proof of~\Cref{thm:existmf}]
	For fixed $g \in M_\gamma$, we denote $f = S_\gamma g$ (i.e. $f(t) = \Phi_g(t,\cdot)\#f^0$). Let us first consider $\gamma\in[-2,0]$ and show $f \in M_\gamma$. The estimate of~\Cref{cor:energyprop} reads
	\[
	\sup_{s\in[0,T]}M_2(f(s))\le F_T\left(\sup_{s\in[0,T]}M_2(g(s))\right)
	\]
	where $F_T$ is the function from~\Cref{lem:expbdd} and the constants are $a = M_2(f^0),\, b = C_\varepsilon$. Moreover, we take the constant $k = 2 M_2(f^0) > a = M_2(f^0)$. This fulfills the criteria of~\Cref{lem:expbdd} so that for $T \le T_C =  \frac{\log 2}{2C_\varepsilon M_2(f^0)}$, we have 
	\[
	\sup_{s\in[0,T]}M_2(f(s)) \le 2M_2(f^0).
	\]
	This proves $f\in M_\gamma$ in the case $\gamma\in[-2,0]$. The case $\gamma\in(-3,-2)$ follows similarly from~\Cref{lem:fLpest} (which replaces~\Cref{cor:energyprop}) and~\Cref{lem:expbdd}.
	
	\ul{Existence and uniqueness:} We now prove that $S_\gamma$ is a contraction on $M_\gamma$ with respect to the metric $d_\infty$. We need to show that there is some universal constant $\kappa\in(0,1)$ such that for every $g^1,g^2\in M_\gamma$, we have
	\[
	d_\infty(S_\gamma g^1, \, S_\gamma g^2) = \sup_{t\in[0,T]}W_\infty(S_\gamma g^1(t), \, S_\gamma g^2(t))\le \kappa d_\infty(g^1, \,g^2).
	\]
	Let us denote $f^i = S_\gamma g^i$ for $i=1, \,2$ so that $f^i\in M_\gamma$ solves~\eqref{eq:contractmap} induced by $g^i$. Let us fix $t\in[0,T]$ and suppress the time dependence. We can use $(\Phi_{g^1}\times \Phi_{g^2})\# f^0$ as an admissible transport plan between $f^1$ and $f^2$ and the following estimate~\cite{S15}
	\begin{align}
		\label{eq:estWinfty}
		\begin{split}
			W_\infty(f^1,f^2) &= \lim_{p\to \infty} W_p(\Phi_{g^1}\#f^0, \Phi_{g^2}\# f^0) \le \lim_{p\to \infty}\left(
			\int_{\R^d}|\Phi_{g^1}(v) - \Phi_{g^2}(v)|^p \, \rd f^0(v)
			\right)^\frac{1}{p}.
		\end{split}
	\end{align}
	The crucial quantity to estimate is the difference between the flow maps. Notice that we are only concerned with the difference for $v\in \text{supp}f^0\subset B_R$ from~\ref{A:cpctsupp}. In particular, for $g^1, g^2 \in M_\gamma$ and $\gamma\in(-2,0]$, we will use the following growth estimate from~\Cref{lem:growthPhi}
	\begin{equation}
		\label{eq:growthPhi}
		\japangle{\Phi_{g^i}} \le \japangle{R}\exp\left\{
		C_\varepsilon M_2(f^0)T
		\right\}, \quad i=1,2.
	\end{equation}
	Using the fact that the flow maps satisfy~\eqref{eq:contODE}, we write again for fixed $t\in[0,T]$
	\begin{align}
		\label{eq:diffflow}
		\begin{split}
			\Phi_{g^1}(v) - \Phi_{g^2}(v) &= \int_0^t U^\varepsilon[g^1](\Phi_{g^1}(s,v)) - U^\varepsilon[g^2](\Phi_{g^2}(s,v))\rd s \\
			&= \int_0^t \underbrace{U^\varepsilon[g^1](\Phi_{g^1}(s,v)) - U^\varepsilon[g^1](\Phi_{g^2}(s,v))}_{=:I_1} \rd s \\
			&\quad + \int_0^t \underbrace{U^\varepsilon[g^1](\Phi_{g^2}(s,v)) - U^\varepsilon[g^2](\Phi_{g^2}(s,v))}_{=:I_2} \rd s.
		\end{split}
	\end{align}
	Starting with the difference $I_1$, we use the Lipschitz regularity of $U^\varepsilon$ from~\Cref{prop:ULipLp} to get
	\[
	|I_1| \le \Lambda_\gamma\left(g^1, \, \Phi_{g^1}(s,v), \, \Phi_{g^2}(s,v)\right) |\Phi_{g^1}(s,v) - \Phi_{g^2}(s,v)|.
	\]
	In particular, absorbing more constants into $C_\varepsilon$, \eqref{eq:growthPhi} gives the bound
	\[
	|\Lambda_\gamma| \le \left\{
	\begin{array}{cl}
		C_\varepsilon M_2(f^0) \japangle{R}^{2+\gamma}\exp\left(
		C_\varepsilon (2+\gamma) M_2(f^0)T
		\right), 	&\gamma\in[-2,0] \\
		C_{\varepsilon, \gamma, p', d}(1 + \|f^0\|_{L^p}), 	&\gamma\in(-3,-2)
	\end{array}
	\right..
	\]
	Turning to $I_2$, we need to estimate the measure-wise difference of the velocity fields $U^\varepsilon$. An application of~\Cref{prop:linest} and~\eqref{eq:growthPhi} gives
	\[
	|I_2| \lesssim_\varepsilon W_\infty(g^1(s), \, g^2(s)) \left\{
	\begin{array}{cl}
		\japangle{R}^{2+\gamma}M_2(f^0)\exp\left(C_\varepsilon (2+\gamma) M_2(f^0)T\right), 	&\gamma\in[-2,0] 	\\
		(1 + \|f^0\|_{L^p}), 	&\gamma\in(-3,-2)
	\end{array}
	\right..
	\]
	Inserting these estimates for $I_1$ and $I_2$ back into~\eqref{eq:diffflow}, we have
	\begin{align*}
		|\Phi_{g^1}(v) - \Phi_{g^2}(v)| &\lesssim_{\varepsilon, \, f^0, \, \gamma, \, T} \int_0^t |\Phi_{g^1}(s,v) - \Phi_{g^2}(s,v)| + W_\infty(g^1(s), g^2(s))\rd s.
	\end{align*}
	Denoting the constant on the right-hand side by $c = c(\varepsilon, f^0)$ (the dependence on $\gamma$ and $T$ is bounded), an application of Gr\"onwall's inequality gives
	\[
	|\Phi_{g^1}(v) - \Phi_{g^2}(v)| \le c \int_0^t e^{c(t-s)}W_\infty(g^1(s), g^2(s))\rd s.
	\]
	Notice that the previous estimate is independent of $v$ and $p>1$, hence when it is substituted into~\eqref{eq:estWinfty}, we obtain
	\begin{align*}
		W_\infty(f^1(t),f^2(t)) &\le c \int_0^t e^{c(t-s)}W_\infty(g^1(s),g^2(s))\rd s \le (e^{cT}-1)d_\infty (g^1,g^2).
	\end{align*}
	By reducing $T>0$ even further (depending only on $\varepsilon$ and $f^0$), we can ensure that the Lipschitz constant $e^{cT} - 1 =: \kappa < 1$. 
	
	\ul{Maximal time of existence:} Having finished with the short time existence and uniqueness of solutions to~\eqref{eq:epslan}, we turn to the statement concerning the maximal time of existence, $T_M$. The dichotomy for $\gamma\in(-3,-2)$ comes from the standard Cauchy-Lipschitz theory. For the case $\gamma\in[-2,0]$, we were able to construct solutions provided we could ensure
	\[
	\sup_{t\in[0,T_M]}M_2(f(t)) \le 2M_2(f^0).
	\]
	On the other hand, the last statement of~\Cref{cor:energyprop} shows that second moments are conserved by solutions of~\eqref{eq:epslan}; $
	M_2(f(t)) = M_2(f^0)$, for every $t\in[0,T_M].$ Thus, we can indefinitely repeat the contraction mapping argument and extend the solution globally to any finite time horizon.
\end{proof}

\section{The Mean Field limit}
\label{sec:mfl}
This section is dedicated to the proof of~\Cref{thm:CCH}. 
The initial computations for both cases $\gamma\in[-2,0]$ and $\gamma\in(-3,-2)$ are the same which we present now until they diverge. To fix notation, let $f$ and $\mu^N = \sum_{i=1}^N m_i \delta_{v^i(t)}$ denote the continuum and (any) empirical solution constructed from~\Cref{sec:existcts,sec:IPS}, respectively. $f^0$ denotes the initial data to $f$ satisfying~\ref{A:cpctsupp} and~\ref{A:gamma} while $\mu_0^N = \sum_{i=1}^N m_i \delta_{v_0^i}$ denotes the initial data of $\mu^N$ satisfying~\ref{B:initmf} and~\ref{B:initcond}. We define the following `discrete' flow
\begin{equation}
	\label{eq:discflow}
	\left\{
	\begin{array}{cl}
		\frac{\rd}{\rd t}F^N(t,s;v) &= U^\varepsilon[\mu^N](F^N(t,s;v)) \\
		F^N(s,s;v)    &= v\in \mathbb{R}^d
	\end{array}
	\right., \quad t, s \in [0,T^N],
\end{equation}
such that $\eta_m(t) = \min_{i\neq j}|v^i(t)-v^j(t)|>0$ for $t\in [0,T^N]$ together with the `continuous' flow
\begin{equation}
	\label{eq:ctsflow}
	\left\{
	\begin{array}{cl}
		\frac{\rd}{\rd t}F(t,s;v)     &= U^\varepsilon[f](F(t,s;v))  \\
		F(s,s;v)     &=v \in \mathbb{R}^d 
	\end{array}
	\right., \quad t, s \in [0,T_m].
\end{equation}
Notice that $T^N>0$ may be taken as \textit{any} arbitrary time horizon for $\gamma\in[-2,0]$ since $\mu^N$ is defined for all times and~\Cref{prop:nocollisionssoft} asserts that $\eta_m^N(t)>0$ for all times. The dependence on $N\in\mathbb{N}$ for the time horizon $T^N$ in~\eqref{eq:discflow} is only relevant for $\gamma\in(-3,-2)$ and this is investigated in the sequel. Here $T_m>0$ is the maximal time of existence of the continuum limit solving~\eqref{eq:epslan} as indicated in~\Cref{thm:existmf}. We may choose $T_m=+\infty$ for $\gamma\in[-2,0]$ while a priori it may be finite for $\gamma\in(-3,-2)$. From the discussion in the~\Cref{sec:existcts,sec:IPS}, the flows in~\eqref{eq:discflow} and~\eqref{eq:ctsflow} are well-defined. Fix $0 < t_0 < \min (T_m, \, T^N)$. Take $\tau^0$ an optimal transport map in $W_\infty$ between $f(t_0)$ and $\mu^N(t_0)$ i.e. $\mu^N(t_0) = \tau^0 \# f(t_0).$ From the construction of $f$ in~\Cref{sec:existcts}, we have that $f(t) = F(t, \,t_0; \, \cdot)\# f(t_0).$ Moreover, we also have $\mu^N(t) = F^N(t, \, t_0; \, \cdot) \# \mu^N(t_0)$. Using a composition of all these maps, we can define a candidate transport map to estimate the $W_\infty$ distance between $f(t)$ and $\mu^N(t)$ by
\[
\tau^t\#f(t) = \mu^N(t), \text{ where } \tau^t = F^N(t,\, t_0; \, \cdot)\, \circ\, \tau^0\, \circ \, F(t_0, \, t; \, \cdot).
\]
For any $1 \le p < \infty$, the $W_p$ Wasserstein distance can be estimated by
\[
W_p^p(\mu^N(t),f(t)) \le \int_{\R^d} |F(t, \, t_0; \, v) - F^N(t, \, t_0; \, \tau^0(v))|^p  \, \rd f_{t_0}(v).
\]
The limit $p = \infty$ is then given by
\[
\eta^N(t) = W_\infty(\mu^N(t),f(t)) \le \| F(t, \, t_0; \, \cdot) - F^N(t, \, t_0;\, \tau^0(\cdot))  \|_{L^\infty(f_{t_0})}.
\]
By definition of the flows defined in~\eqref{eq:discflow} and~\eqref{eq:ctsflow}, we have
\[
\left.\frac{\rd}{\rd t}\right|_{t=t_0^+} F^N(t, \, t_0;\, \tau^0(v)) - F(t, \, t_0; \, v) = U^\varepsilon[\mu^N(t_0)](\tau^0(v)) - U^\varepsilon[f(t_0)](v).
\]
Therefore, we can estimate
\begin{equation}
	\label{eq:etagrowth}
	\left.\frac{\rd}{\rd t}\right|_{t=t_0^+}\eta^N(t) \le \| U[\mu^N(t_0)](\tau^0(v)) - U[f(t_0)](v) \|_{L_v^\infty(f_{t_0})}.
\end{equation}
We expand the velocity difference in~\eqref{eq:etagrowth} using the fact that $\tau^0$ transports $f_{t_0}$ to $\mu_{t_0}^N$
\begin{align}
	\label{eq:diffvel}
	\begin{split}
		&\quad U^\varepsilon[\mu^N(t_0)](\tau^0(v)) - U^\varepsilon[f(t_0)](v) \\ &= \int_{\R^d} K_{\mu^N(t_0)}(\tau^0(v),w) \,d\mu_{t_0}^N(w) - \int_{\R^d}K_{f(t_0)}(v,w)\, \rd f_{t_0}(w) \\
		&= \int_{\R^d} (K_{\mu^N(t_0)}(\tau^0(v), \tau^0(w)) - K_{f(t_0)}(v,w)) \, \rd f_{t_0}(w) \\
		&= \int_{\R^d}\underbrace{\left\{
			K_{\mu_{t_0}^N}(\tau^0(v),\tau^0(w)) - K_{\mu_{t_0}^N}(\tau^0(v),w) + K_{\mu_{t_0}^N}(\tau^0(v),w) - K_{\mu_{t_0}^N}(v,w)
			\right\}}_{=:D_1} \, \rd f_{t_0}(w) \\
		&\quad + \int_{\R^d}\underbrace{\left\{
			K_{\mu_{t_0}^N}(v,w) - K_{f_{t_0}}(v,w)
			\right\}}_{=:D_2}\, \rd f_{t_0}(w).
	\end{split}
\end{align}
Starting with the $D_2$ term, we need only concern ourselves with bounded $v$ and $w$ in the integrations owing to~\Cref{lem:growthPhi}. In particular, the regions of integration can be restricted to
\begin{equation}
	\label{eq:bddvel}
	|v|,\, |w| \le \left\{
	\begin{array}{cl}
		R \exp\left(C_\varepsilon t_0\right), & \gamma\in(-2,0] 	\\
		R + C_\varepsilon t_0, 	&\gamma\in(-3,-2)
	\end{array}
	\right..
\end{equation}
By~\Cref{lem:diffKg,lem:peetre}, we can estimate the $D_2$ term by
\begin{align}
	\label{eq:D2modsoft}
	\begin{split}
		&\quad \int_{\R^d} |D_2| \, \rd f_{t_0}(w) \lesssim_\varepsilon W_\infty(\mu_{t_0}^N, f_{t_0})\int \min\left(|v-w|^{3+\gamma}, \, |v-w|^{2+\gamma}\right) \, \rd f_{t_0}(w) 	\\
		&\le W_\infty(\mu_{t_0}^N, f_{t_0})\times \left\{
		\begin{array}{cl}
			R^{2+\gamma}\exp\left(C_\varepsilon(2+\gamma) t_0\right) M_{2+\gamma}(f_{t_0}),	&\gamma\in(-2,0] 	\\
			1, 	&\gamma\in(-3,-2)
		\end{array}
		\right..
	\end{split}
\end{align}
As for the $D_1$ term, we can complete the proof of~\Cref{thm:CCH} for $\gamma\in(-2,0]$.
\begin{proof}[Proof of~\Cref{thm:CCH} for moderately soft potentials]
	Building on the previous discussion, we only need to estimate the $D_1$ term. Using~\Cref{lem:contourestimate} twice, we get
	\begin{align*}
		&\quad \int_{\R^d} |D_1| \, \rd f_{t_0}(w) \\
		&\lesssim_\varepsilon \int_{B_{R\exp\left(C_\varepsilon t_0\right)}} |\tau^0(w) - w|\min\left(
		|\tau^0(v) - \tau^0(w)|^{2+\gamma}, |\tau^0(v) - w|^{2+\gamma}
		\right) \, \rd f_{t_0}(w) \\
		&\quad +  |\tau^0(v) - v| \int_{B_{R\exp\left(C_\varepsilon t_0\right)}} \min\left(
		|\tau^0(v) - w|^{2+\gamma}, |v - w|^{2+\gamma}
		\right) \, \rd f_{t_0}(w).
	\end{align*}
	The assumptions~\ref{B:initmf} and~\ref{A:cpctsupp} say $W_\infty(\mu_0^N, f_0)\to 0$ and supp$f_0\subset B_R$, thus for sufficiently large $N\gg 1$, we must have supp$\mu_0^N\subset B_{R+1}$. Moreover, $\tau^0$ pushes forward $f_{t_0}$ to $\mu_{t_0}^N$, so we obtain Im$\tau^0\subset B_{(R+1)\exp(C_\varepsilon t_0)}$ from~\eqref{eq:bddvel}. Hence, we can bluntly estimate the minimum terms using~\Cref{lem:peetre} to obtain
	\begin{align*}
		&\quad \int_{\R^d} |D_1| \, \rd f_{t_0}(w) 	\\
		&\le R^{2+\gamma} \exp(C_\varepsilon(2+\gamma)t_0) \left(\int |\tau^0(w) - w|\, \rd f_{t_0}(w) + |\tau^0(v) - v|\right) \\
		&\le 2 R^{2+\gamma} \exp(C_\varepsilon(2+\gamma)t_0) W_\infty(\mu_{t_0}^N, f_{t_0}). 
	\end{align*}
	Collecting this and~\eqref{eq:D2modsoft}, plugging them into~\eqref{eq:diffvel} and then~\eqref{eq:etagrowth}, we arrive at
	\[
	\left.\frac{\rd}{\rd t}\right|_{t=t_0^+} \eta^N(t) \lesssim_\varepsilon R^{2+\gamma}e^{C_\varepsilon (2+\gamma) t_0}\eta^N(t_0).
	\]
	As $t_0 \in(0,T)$ was chosen arbitrarily, a direct application of Gr\"onwall's inequality gives
	\[
	W_\infty(\mu^N(t), f(t)) \le W_\infty(\mu_0^N, f^0)\exp \left\{
	C_\varepsilon^1 R^{2+\gamma} e^{C_\varepsilon (2+\gamma) T}t
	\right\}, \quad \forall t\in[0,T].
	\]
	This implies the mean field limit for $\gamma\in(-2,0]$.
\end{proof}

\begin{proof}[Proof of~\Cref{thm:CCH} for very soft potentials]
	For $\gamma\in(-3,-2)$, the same method to estimate the $D_1$ term in~\eqref{eq:diffvel} does not work. Moreover, the construction of the particle solutions $\mu^N$ is only local in time up to some time horizon $T^N>0$ (c.f.~\Cref{sec:IPS}) which may be strictly less than $T_m$ and may also degenerate to 0 as $N\to+\infty$. We overcome these issues to show the mean field limit by repeating the inter-particle distance analysis from~\cite{CCH14}. The general steps from~\cite{CCH14} to show $\eta^N(t) \to 0$ are to couple the evolutions of $\eta^N$ and $\eta_m^N$ together in the following way where we recall $\eta_m^N$ denotes the inter-minimum particle distance for the particles in $\mu^N$
	\[
	\eta_m^N(t) := \min_{i\neq j}|v^i(t) - v^j(t)|.
	\]
	\begin{enumerate}
		\item We show first in~\Cref{sec:step1} the growth estimate of $\eta$ coupled with $\eta_m$
		\begin{equation}
			\label{eq:growtheta}
			\frac{\rd}{\rd t}\eta^N \lesssim_{\varepsilon,\gamma} \eta^N (1+\| f\|_{L^p})\left(
			1 + \left(\eta^N\right)^\frac{d}{p'}\left(\eta_m^N\right)^{1+\gamma}
			\right).
		\end{equation}
		\item Then in~\Cref{sec:step2} we obtain the decay estimate of $\eta_m$ coupled with $\eta$ 
		\begin{equation}
			\label{eq:decayetam}
			\frac{\rd}{\rd t}\eta_m^N \gtrsim_{\varepsilon,\gamma}-\eta_m^N(1+\| f\|_{L^p})\left(
			1 + \left(\eta^N\right)^\frac{d}{p'}\left(\eta_m^N\right)^{1+\gamma}
			\right).
		\end{equation}
		\item \label{step3} The coupled system~\eqref{eq:growtheta} and~\eqref{eq:decayetam} together with~\ref{B:initcond} allow us to deduce both $\liminf_{N\to\infty}T^N\ge T_m$ and $\eta^N(t) \to 0$ for all times $t\in[0,T_m)$. This is performed in~\Cref{sec:step3} based on the argument in~\cite{CCH14}. 
	\end{enumerate}
\end{proof}
Since the velocity vector field $U^\varepsilon$ depends on the solution, we cannot directly repeat the arguments from~\cite{CCH14} to establish~\eqref{eq:growtheta} and~\eqref{eq:decayetam}. Nevertheless, once these estimates are proven, step~\ref{step3} follows exactly as in~\cite{CCH14} which we leave to~\Cref{sec:step3} for completeness.
\subsection{Step 1}
\label{sec:step1}
We focus on the $D_1$ term from~\eqref{eq:diffvel} recalling that~\eqref{eq:D2modsoft} implies
\[
\int_{\R^d}|D_2|\, \rd f_{t_0}(w) \lesssim_\varepsilon \eta^N (t_0).
\]

\begin{proposition}
	\label{prop:estd1int}
	For fixed $\varepsilon>0$ and $\gamma\in(-3,-2)$, we have the estimate
	\[
	\int_{\R^d}|D_1| \, \rd f_{t_0}(w) \lesssim_{\varepsilon,\gamma,p'} \eta^N (1 + \| f_{t_0}\|_{L^p})(1 + \left(\eta^N\right)^\frac{d}{p'}\left(\eta_m^N\right)^{1+\gamma}).
	\]
\end{proposition}
Substituting these estimates for $D_1$ and $D_2$ into~\eqref{eq:diffvel} and then~\eqref{eq:etagrowth} gives~\eqref{eq:growtheta} completing the first step.
\begin{proof}
	Using~\eqref{lem:contourestimate}, we obtain
	\begin{align*}
		|D_1| &\lesssim_{\varepsilon,\gamma}|\tau^0(w) - w|\max\left(
		|\tau^0(v) - w|^{2+\gamma}, \, |\tau^0(v) - \tau^0(w)|^{2+\gamma}
		\right) \\
		&\quad + |\tau^0(v) - v| \max \left(
		|\tau^0(v) - w|^{2+\gamma}, \, |v-w|^{2+\gamma}
		\right).
	\end{align*}
	\ul{Integration region $|v-w| \ge 4\eta^N$:} We first deduce
	\[
	|\tau^0(v) - \tau^0(w)| \ge |v-w| - |\tau^0(v) - v| - |\tau^0(w) - w| \ge |v-w| - 2 \eta^N \ge \frac{|v-w|}{2}.
	\]
	The second inequality is obtained by remembering $\tau^0$ is an optimal transport map in $W_\infty$ between $f(t_0)$ and $\mu^N(t_0)$. Similarly, we have the estimate
	\[
	|\tau^0(v) - w| \ge |v-w| - |\tau^0(v) - v| \ge |v-w| - \eta^N \ge \frac{3|v-w|}{4}.
	\]
	Overall, these estimates lead to
	\[
	|D_1| \lesssim_{\varepsilon,\gamma}\eta^N |v-w|^{2+\gamma}
	\]
	and integrating over $\{w\in \R^d \, | \, |v-w|\ge 4\eta^N\}$ yields
	\begin{align*}
		\int_{|v-w|\ge 4\eta^N}|D_1|\, \rd f_{t_0}(w) &\lesssim_{\varepsilon,\gamma}\eta^N \left(
		\int_{4\eta^N \le |v-w| \le 1} + \int_{|v-w| > 1}
		\right)|v-w|^{2+\gamma}\, \rd f_{t_0}(w) \\
		&\lesssim_{p',\gamma}\eta^N (\|f_{t_0}\|_{L^p} + 1).
	\end{align*}
	\ul{Integration region $|v-w| < 4\eta^N$:} Here, we do not use the cancellations in
	\[
	D_1 = K_{\mu^N}(\tau^0(v),\tau^0(w)) - K_{\mu^N}(v,w),
	\]
	instead, we estimate each term using~\Cref{prop:boundsKg}. Since Im$\tau^0\subset \{v^i\}_{i=1}^N$, if $\tau^0(v) = \tau^0(w)$, then the H\"older regularity of $\log [\mu^N*G^\varepsilon]$ (after interpolating the estimates in~\eqref{eq:extlogdiffsgeq1}) gives $K_{\mu^N}(\tau^0(v),\tau^0(w)) = 0$. Otherwise, we use $|\tau^0(v) - \tau^0(w)| \ge \eta_m$ and~\Cref{prop:boundsKg} to deduce
	\[
	|D_1| \lesssim_\varepsilon \left(\eta_m^N\right)^{2+\gamma}+|v-w|^{2+\gamma} .
	\]
	By H\"older's inequality, the integral can be estimated by
	\begin{align*}
		\int_{|v-w|< 4\eta^N}|D_1| \, \rd f_{t_0}(w) &\lesssim_\varepsilon \int_{|v-w|< 4\eta^N} |v-w|^{2+\gamma}\, \rd f_{t_0}(w) + \left(\eta_m^N\right)^{2+\gamma}\int_{|v-w|< 4\eta^N}\, \rd f_{t_0}(w) \\
		&\le \|f_{t_0}\|_{L^p}\left(
		\left(
		\int_{|v-w|< 4\eta^N} |v-w|^{(2+\gamma)p'}\, \rd w
		\right)^\frac{1}{p'} + \left(
		\int_{|v-w|< 4\eta^N} 1\, \rd w
		\right)^\frac{1}{p'} \left(\eta_m^N\right)^{2+\gamma}
		\right) \\
		&\lesssim_{\gamma,p'} \|f_{t_0}\|_{L^p}\left(\left(\eta^N\right)^{\frac{d}{p'}+ 2+ \gamma} + \left(\eta_m^N\right)^{2+\gamma}\left(\eta^N\right)^\frac{d}{p'}
		\right).
	\end{align*}
	Finally, choose two indices $i, \, j$ such that $\eta_m^N = |v^i - v^j|$. We seek to estimate $\eta_m$ against $\eta$ by looking at where $\tau^0$ sends the midpoint $\frac{v^i+v^j}{2}$. In the case that $\frac{v^i+v^j}{2} \notin \text{supp}f$, then we can define $\tau^0$ to be whatever we want as it does not affect the $W_\infty$ distance. In particular, we can assign $\tau^0(v) \in \{v^i\}_{i=1}^N$ for every $v\notin$ supp$f^0$ without changing the transport cost. Suppose
	\[
	\tau^0\left(
	\frac{v^i+v^j}{2}
	\right) = v^k \in \{v^i\}_{i=1}^N.
	\]
	Without loss of generality $v_k\ne v_i$, and we have the following lower bound
	\begin{align*}
		&\quad \eta^N \ge \left|
		\tau^0\left(
		\frac{v^i+v^j}{2}
		\right) - \frac{v^i+v^j}{2}
		\right| = \left|v^k - \frac{v^i+v^j}{2}\right| = \left|v^k - v^i + \frac{v^i-v^j}{2}\right| \\
		&\ge |v^k-v^i| - \frac{1}{2}|v^i-v^j| \ge \frac{1}{2}\eta_m^N.
	\end{align*}
	This implies $\eta_m^N \le 2\eta^N$ which simplifies
	\[
	\int_{|v-w|< 4\eta^N}|D_1| \, \rd f_{t_0}(w) \lesssim_{\varepsilon,\gamma,p'} \left(\eta^N\right)^{\frac{d}{p'}+1}\left(\eta_m^N\right)^{1+\gamma}\|f_{t_0}\|_{L^p}.
	\]
\end{proof}
\subsection{Step 2}
\label{sec:step2}
Having derived an upper bound for the growth of $\eta^N$ coupled with $\eta_m^N$, we need to find a corresponding lower bound for the decrease of $\eta_m^N$ coupled with $\eta^N$ to close the system.
\begin{proposition}
	The minimum inter-particle distance satisfies the lower bound for its decay
	\[
	\frac{\rd}{\rd t}\eta_m^N \gtrsim_{\varepsilon,\gamma}- \eta_m^N(1+ \|f \|_{L^p})(1 + \left(\eta^N\right)^\frac{d}{p'}\left(\eta_m^N\right)^{1+\gamma}).
	\]
\end{proposition}
\begin{proof}
	Choose two indices $i, \, j =1,\dots, N$ such that $|v^i - v^j| = \eta_m^N$ where we will suppress time dependence for simplicity. We have
	\begin{align*}
		\frac{\rd}{\rd t}|v^i - v^j| &\ge -|U^\varepsilon[\mu^N](v^i) - U^\varepsilon[\mu^N](v^j)| \\
		&\ge - \int_{\R^3}|K_{\mu^N}(v^i,w) - K_{\mu^N}(v^j,w)| \,\rd \mu^N(w) \\
		&= - \int_{\R^3} |K_{\mu^N}(v^i,\tau(w)) - K_{\mu^N}(v^j,\tau(w))|\, \rd f(w).
	\end{align*}
	Here, we have set $\tau$ as an optimal transfer map in $W_\infty$ such that $\mu^N(t) = \tau\# f(t)$ for $t \in [0,\min(T,T^N))$. We split the integration into the following domains
	\[
	\mathcal{A} = \{ w \, : \, \min(|v^i-w|,\, |v^j-w|)\ge 2\eta^N\}, \quad \mathcal{B} = \R^d \setminus \mathcal{A}.
	\]
	Starting with $\mathcal{A}$, we use the inequality
	\[
	|v^i - \tau(w)| \ge |v^i-w| - |w - \tau(w)| \ge |v^i - w| - \eta^N \ge \frac{|v^i-w|}{2}
	\]
	and~\Cref{lem:contourestimate} to deduce
	\begin{align*}
		&\quad \int_{\mathcal{A}}|K_{\mu^N}(v^i,\tau(w)) - K_{\mu^N}(v^j,\tau(w))|\, \rd f(w) \\ &\lesssim_{\varepsilon,\gamma}|v^i - v^j| \int_{\mathcal{A}}\max \left(
		|v^i - \tau(w)|^{2+\gamma}, |v^j - \tau(w)|^{2+\gamma}
		\right)\, \rd f(w) \\
		&\le 2^{-(2+\gamma)}|v^i-v^j| \int_{\mathcal{A}}(|v^i - w|^{2+\gamma} + |v^j - w|^{2+\gamma})\, \rd f(w) \\
		&\lesssim_{\gamma}\eta_m^N (1 + \|f\|_{L^p}).
	\end{align*}
	In the last line, we have bluntly estimated
	\begin{align*}
		\int_{\mathcal{A}}|v^i - w|^{2+\gamma}\, \rd f(w) &\le \int_{\R^d}|v^i - w|^{2+\gamma}\, \rd f(w) \\ &\le \int_{|v^i - w| \ge 1} \, \rd f(w) + \int_{|v^i - w| <1 }|v^i - w|^{2+\gamma}\, \rd f(w).
	\end{align*}
	with the usual H\"older's inequality for the second term and similarly for $v^j$.
	
	Turning to the region $\mathcal{B}$, since Im$\tau \subset \{v^i\}_{i=1}^N$, as soon as $v^i \neq \tau(w),$ we must have
	\[
	|v^i - \tau(w)| \ge \eta_m^N,
	\]
	with a similar estimate for $v^j$. By further blunting the $L^\infty$ estimate in~\Cref{prop:boundsKg}, we obtain
	\begin{align*}
		|K_{\mu^N}(v^i,\tau(w))| &\lesssim_\varepsilon|v^i-\tau(w)|^{2+\gamma} \le \left(\eta_m^N\right)^{2+\gamma}.
	\end{align*}
	If $v^i=\tau(w)$, then the H\"older regularity of $\nabla \log [\mu^N*G^\varepsilon]$ from~\eqref{eq:extlogdiffsgeq1} gives $K_{\mu^N}(v^i,\tau(w)) = 0$. The familiar method using H\"older's inequality gives
	\[
	\int_{\mathcal{B}}\, \rd f(w) \le \left(
	\int_{\mathcal{B}}\, \rd w
	\right)^\frac{1}{p'}\|f\|_{L^p} \lesssim \left(\eta^N\right)^\frac{d}{p'}\|f\|_{L^p}.
	\]
	Putting these two estimates together, we treat the full integral over $\mathcal{B}$ by
	\begin{align*}
		\int_{\mathcal{B}}|K_{\mu^N}(v^i,\tau(w)) - K_{\mu^N}(v^j,\tau(w))|\, \rd f(w) &\lesssim \left(\eta_m^N\right)^{2+\gamma} \int_{\mathcal{B}}\, \rd f(w) \lesssim \left(\eta^N\right)^\frac{d}{p'}\left(\eta_m^N\right)^{2+\gamma} \|f\|_{L^p}.
	\end{align*}
	Finally, we add up the integrals over $\mathcal{A}$ and $\mathcal{B}$ to get
	\begin{align*}
		\int_{\R^d}|K_{\mu^N}(v^i,\tau(w)) - K_{\mu^N}(v^j,\tau(w))|\, \rd f(w) &\lesssim_{\varepsilon,\gamma} \eta_m^N(1+ \|f \|_{L^p})(1 + \left(\eta^N\right)^\frac{d}{p'}\left(\eta_m^N\right)^{1+\gamma}).
	\end{align*}
\end{proof}

\section*{Acknowledgements}
JAC was supported the Advanced Grant Nonlocal-CPD (Nonlocal PDEs for Complex Particle Dynamics: 	Phase Transitions, Patterns and Synchronization) of the European Research Council Executive Agency (ERC) under the European Union's Horizon 2020 research and innovation programme (grant agreement No. 883363). JAC was also partially supported by the EPSRC grant numbers EP/T022132/1 and EP/V051121/1. MGD was partially supported by NSF-DMS-2205937 and NSF-DMS RTG 1840314. JW was supported by the the Mathematical Institute Award of the University of Oxford. The authors would like to thank the Isaac Newton Institute for Mathematical Sciences, Cambridge, for support and hospitality during the programme \textit{Frontiers in Kinetic Theory} where work on this paper was undertaken. This work was supported by EPSRC grant no EP/R014604/1.

\begin{appendices}
	\section{Proof of~\Cref{lem:contourestimate}}
	\label{sec:contourproof}
	The structure of our kernel is more general than those considered in~\cite{CCH14}, but the idea is the same and we provide the details for completeness.
	
	\ul{The case $\gamma\in[-2,0]$:} The fundamental theorem of calculus with~\Cref{prop:boundsKg} give
	\begin{align*}
		|K_g(v_1,w) - K_g(v_2,w)| &= \left|
		\int_0^1 \frac{\rd}{\rd t}K_g(tv_1 + (1-t)v_2,w)\rd t
		\right| \\
		&\lesssim_\varepsilon |v_1-v_2|\int_0^1 |tv_1 + (1-t)v_2 - w|^{2+\gamma}\rd t.
	\end{align*}
	Up to a constant depending on $\gamma$, the integrand can be estimated by
	\begin{align*}
		|tv_1 + (1-t)v_2 - w|^{2+\gamma} &\lesssim_\gamma |v_1 - w|^{2+\gamma} + |v_2 - w|^{2+\gamma} \\ &\lesssim \max\left(|v_1-w|^{2+\gamma},|v_2-w|^{2+\gamma}\right).
	\end{align*}
	\ul{The case $\gamma\in[-3,-2)$:} Set $\Gamma(t) = (1-t)v_1 + tv_2 -w$ and we separate into further cases.
	
	\ul{Case 1 - For every $t\in[0,1]$, we have $|\Gamma(t)| \ge \frac{1}{4}\min(|v_1-w|, |v_2 - w|)$:} We can repeat the previous computations almost exactly and recover the desired estimate.
	
	\ul{Case 2 - There is a $t\in[0,1]$ such that $|\Gamma(t)| < \frac{1}{4}\min(|v_1-w|, |v_2 - w|)$:} We need to perturb the original contour $\Gamma$ to avoid the possible singularity. Notice that we can find $t\in(0,1)$ such that $|\Gamma(t)| < \frac{1}{4}\min(|v_1-w|, |v_2 - w|)$. We first take (the unique) $t_m\in(0,1)$ such that
	\[
	|\Gamma(t_m)| = \min_{t\in [0,1]}|\Gamma(t)|.
	\]
	Next, define the other two time points where $|\Gamma(t)| = \frac{1}{4}\min(|v_1-w|, |v_2-w|),$
	\begin{align*}
		t_i &:= \inf\left\{
		t \in [0,1] \, : \, |\Gamma(t)| = \frac{1}{4}\min(|v_1-w|, |v_2-w|)
		\right\}, \\
		t_s &:= \sup\left\{
		t \in [0,1] \, : \, |\Gamma(t)| = \frac{1}{4}\min(|v_1-w|, |v_2-w|)
		\right\}.
	\end{align*}
	By continuity of $|\Gamma(t)|$, we have that all $t_m, \, t_i, \, t_s \in (0,1).$ The triangle formed by connecting the vectors $\Gamma(t_i), \, \Gamma(t_s) - \Gamma(t_i),$ and $\Gamma(t_s)$ is \textit{isosceles} so the following quantity is well-defined (see~\Cref{fig:contour})
	\[
	r := |\Gamma(t_i) - \Gamma(t_m)| = |\Gamma(t_m) - \Gamma(t_s)|.
	\]
	\begin{figure}[H]
		\centering
		\caption{Simplistic visual perturbation of $\Gamma(t)$ to avoid the singularity.}
		\label{fig:contour}
		\begin{tikzpicture}
			\draw[very thin, opacity = 0.5, gray, help lines] (-5,-5) grid (5,5);
			\draw[thick] (-5,0) -- (5,0);
			\filldraw[black] (-5,0) circle (1.5pt) node[left] {$v_1-w$};
			\filldraw[black] (5,0) circle (1.5pt) node[right] {$v_2-w$};
			\filldraw[black] (0,-4) circle (1.5pt) node[below] {0};
			\filldraw[black] (0,0) circle (1.5pt) node[above] {$\Gamma(t_m)$};
			\draw[dashed, blue, <->] (0,-4) -- (0,0) node[pos = 0.5, right] {$\min |\Gamma|$};
			\filldraw[black] (3,0) circle (1.5pt) node[below] {$\Gamma(t_s)$};
			\filldraw[black] (-3,0) circle (1.5pt) node[below] {$\Gamma(t_i)$};
			\draw[YellowGreen] (0:3cm) arc (0:180:3cm);
			\draw[dashed, ->, Rhodamine] (0,-4) -- (-3,0) node[pos = 0.5, below, left] {$\frac{1}{4}\min (|v_1-w|, |v_2-w|)$};
			\draw[dashed, ->, Rhodamine] (0,-4) -- (3,0) node[pos = 0.5, below, right] {$\frac{1}{4}\min (|v_1-w|, |v_2-w|)$};
			\draw[<->, Orchid, dashed] (-3,3.5) -- (0,3.5) node[pos = 0.5, above] {$r =|\Gamma(t_i) - \Gamma(t_m)|$};
			\draw[->, Sepia] (0,0) -- (135:3cm) node[pos=1, above, left] {$A(\theta)$};
			\draw[->, Sepia] (180:1cm) arc (180:135:1cm) node[pos=0.5, left] {$\theta$}; 
		\end{tikzpicture}
	\end{figure}
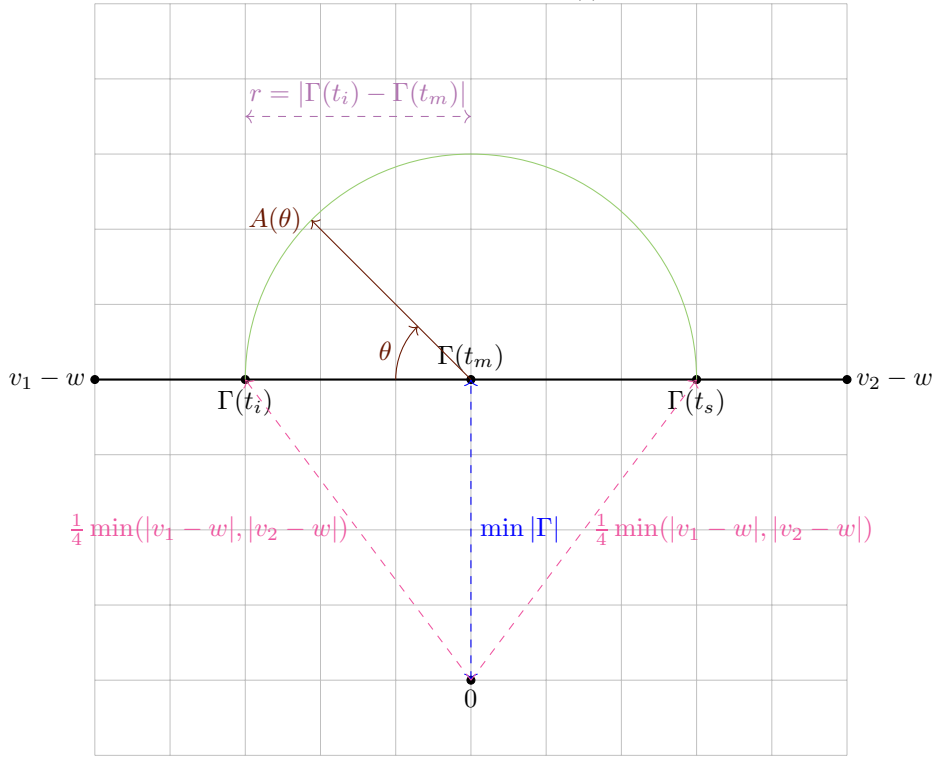
	We wish to apply the fundamental theorem by taking the contour connecting $v_1-w$ to $v_2-w$ that traces a semicircular arc from $\Gamma(t_i)$ to $\Gamma(t_s)$ in the direction furthest from the origin (the \textcolor{YellowGreen}{green} arc in~\Cref{fig:contour}). More precisely, the direction furthest away from the origin is defined as
	\[
	e := \frac{\Gamma(t_m)}{|\Gamma(t_m)|},
	\]
	or if $\Gamma(t_m) = 0$, take any $e\in \mathbb{S}^{d-1}$. For a given angle $\theta \in [0,\pi]$, the \textcolor{YellowGreen}{green} arc can be parameterised by
	\[
	A(\theta) := \Gamma(t_m) + r\left(\cos\theta \frac{\Gamma(t_i) - \Gamma(t_m)}{r} + \sin \theta \, e\right).
	\]
	Observe that by the (reverse) triangle inequality and the fact that $e \perp \Gamma(t_i) - \Gamma(t_m)$, we have the lower bound for all $\theta \in [0,\pi]$
	\begin{equation}
		\label{eq:Alowbdd}
		|A(\theta)| \ge |\Gamma(t_m)| - r = |\Gamma(t_m)| - |\Gamma(t_i) - \Gamma(t_m)| \ge |\Gamma(t_i)| = \frac{1}{4}\min (|v_1-w|,|v_2-w|).
	\end{equation}
	Putting these pieces together, we define the perturbed contour $\tilde{\Gamma} : [0, \,t_i + \pi + 1-t_s] \to \R^d$ by
	\[
	\tilde{\Gamma}(t) := \left\{
	\begin{array}{cl}
		\Gamma(t),     &t\in [0, \,t_i]  \\
		A(t - t_i),     &t\in [t_i, \, t_i+\pi]  \\
		\Gamma(t - t_i - \pi + t_s),  &t\in[t_i+\pi, \, t_i + \pi + 1-t_s]
	\end{array}
	\right..
	\]
	We will apply the fundamental theorem of calculus on each of the three pieces of $\tilde{\Gamma}$ to estimate the difference
	\begin{align}
		\label{eq:diffcont}
		\begin{split}
			|K_g(v_1,w) - K_g(v_2,w)| &\le |K_g(\tilde{\Gamma}(0) + w, w) - K_g(\tilde{\Gamma}(t_i) + w,w)| \\
			&\quad + |K_g(\tilde{\Gamma}(t_i) + w, w) - K_g(\tilde{\Gamma}(t_i + \pi) + w, w)| \\
			&\quad + |K_g(\tilde{\Gamma}(t_i + \pi)+w,w) - K_g(\tilde{\Gamma}(t_i+\pi + 1-t_s) + w,w)| \\
			&\le \int_0^{t_i}\left|
			\frac{\rd}{\rd t}K_g(\tilde{\Gamma}(t) + w,w)
			\right|\, \rd w \\
			&\quad +\int_{t_i}^{t_i+\pi}\left|
			\frac{\rd}{\rd t}K_g(\tilde{\Gamma}(t) + w,w)
			\right|\, \rd w \\
			&\quad +\int_{t_i+\pi}^{t_i+\pi+1-t_s}\left|
			\frac{\rd}{\rd t}K_g(\tilde{\Gamma}(t) + w,w)
			\right|\, \rd w \\
			&=: T_1 + T_2 + T_3.
		\end{split}
	\end{align}
	Starting with $T_1$, the chain rule gives
	\begin{align*}
		T_1 &\le |v_1-v_2|\int_0^{t_i} |\nabla_v K_g(\tilde{\Gamma}(t) + w,w)|\, \rd t.
	\end{align*}
	Using the derivative estimate in~\Cref{prop:boundsKg} and the fact that $|\Gamma(t)| \ge \frac{1}{4}\min(|v_1-w|,|v_2-w|)$ for $t\in [0,t_i]$, we obtain
	\begin{align}
		\label{eq:estT1cont}
		\begin{split}
			T_1 &\lesssim_\varepsilon |v_1-v_2|\int_0^{t_i} |\tilde{\Gamma}(t)|^{2+\gamma}dt \le |v_1-v_2| \int_0^{t_i}\max(|v_1-w|^{2+\gamma}, |v_2-w|^{2+\gamma})dt \\
			&\le t_i |v_1-v_2| \max(|v_1-w|^{2+\gamma}, |v_2-w|^{2+\gamma}).
		\end{split}
	\end{align}
	Similarly for $T_3$, we have
	\begin{equation}
		\label{eq:estT2cont}
		T_3 \lesssim_\varepsilon (1-t_s) |v_1-v_2| \max(|v_1-w|^{2+\gamma}, |v_2-w|^{2+\gamma}).
	\end{equation}
	We now turn to $T_2$, we substitute $A(t-t_i)$ into this piece and use the derivative estimate from~\Cref{prop:boundsKg} with the chain rule to get
	\[
	T_2 \lesssim_\varepsilon \int_{t_i}^{t_i+\pi} \left|\frac{\rd}{\rd t}A(t-t_i)
	\right||A(t-t_i)|^{2+\gamma}\rd t.
	\]
	Recalling the definitions of $r$ (this is the length of a particular segment of $[v_1-w,v_2-w]$) and $A$ together with the lower bound~\eqref{eq:Alowbdd}, we use
	\[
	\left|\frac{\rd}{\rd t}A\right| = r \le |v_1-v_2| \quad \text{and} \quad |A| \ge \frac{1}{4}\min(|v_1-w|,|v_2-w|)
	\]
	so that we have
	\[
	T_2 \lesssim_\varepsilon \pi |v_1-v_2| \max(|v_1-w|^{2+\gamma}, |v_2-w|^{2+\gamma}).
	\]
	Putting this inequality with~\eqref{eq:estT2cont} and \eqref{eq:estT1cont} into~\eqref{eq:diffcont}, we achieve the desired result.
	\section{Proof of~\Cref{lem:fLpest}}
	\label{sec:fLpest}
	By our abuse of notation from interchanging probability measures with their densities, we write down the explicit formula for $f(t,v)$ as a density
	\begin{equation}
		\label{eq:fcov}
		f(t,v) = \frac{f^0(\Phi_g^{-1}(t,v))}{|\text{det}(\nabla \Phi_g(t,\Phi_g^{-1}(t,v)))|}.
	\end{equation}
	Here, the inverse $\Phi_g^{-1}$ should be thought of as the `reverse' flow map to $\Phi_g$ where the direction of time has been reversed. Changing variables with~\eqref{eq:fcov}, we have
	\begin{equation}
		\label{eq:fLp}
		\quad \int |f(t,v)|^p\, \rd v = \int \frac{|f^0(v)|^p}{|\text{det}(\nabla \Phi_g(t,v))|^{p-1}}\, \rd v.
	\end{equation}
	We turn to estimating the denominator in the integrand of~\eqref{eq:fLp}. Again, standard facts about the flow map $\Phi_g$ from~\cite{G03} give the following formula
	\begin{equation*}
		\text{det}\nabla \Phi_g(t,v) = \exp\left\{
		\int_0^t \nabla_v\cdot U^\varepsilon[g](\Phi_g(s,v))\, \rd s
		\right\}.
	\end{equation*}
	From an application of the Dominated Convergence Theorem and~\Cref{prop:boundsKg}, we have
	\begin{align*}
		&\quad |\nabla_v\cdot U^\varepsilon[g](v)| \le \int |\nabla_v\cdot K_g(v,w)| \, \rd g(w) \\
		&\lesssim_\varepsilon \int |v-w|^{2+\gamma}\, \rd g(w) \le 1 + C_{\gamma, d}\|g\|_{L^p}.
	\end{align*}
	The last computation is obtained by the usual method of splitting the integration region between $|v-w| < 1$ and $|v-w| \ge 1$ recalling $2+\gamma < 0$. Inserting this inequality into~\eqref{eq:fLp}, we obtain the desired estimate
	\[
	\int |f(t,v)|^p \rd v \le \left(\int |f^0(v)|^p \rd v\right) \exp \left\{
	C_{\varepsilon, \gamma, d}(p-1)  \left(
	1 + \mathrm{esssup}_{s\in[0,T]}\|g(s)\|_{L^p}
	\right)t
	\right\}.
	\]
	Finally, \Cref{cor:energyprop} already proved $f\in C([0,T];\mathscr{P}_c(\R^d))$ and the $L_t^\infty L_v^p$ property is clear from the estimate we have just proved.
	\section{The interacting particle system}
	\label{sec:IPS}
	This section is concerned with proving~\Cref{lem:existparticle}; the well-posedness of the particle system described in~\eqref{eq:epsparticle}. Throughout this section, the number $N\in\mathbb{N}$ of particles is fixed as well as the positive weights $\{m_i\}_{i=1}^N$ and initial points $\{v_0^i\}_{i=1}^N$. We denote the initial empirical data by $\mu_0^N = \sum_{i=1}^N m_{i,N}\delta_{v_0^i}$. We can apply the same arguments from~\Cref{sec:mfl} for $\gamma\in(-2,0]$.
	\begin{proof}[Proof of~\Cref{lem:existparticle} for $-2< \gamma\le 0$]
		The initial empirical data $\mu_0^N$ satisfies~\ref{A:cpctsupp} with radius of support $R_N := \max_{i=1,\dots,N}|v_0^i|$. Applying~\Cref{thm:existmf} for any $T>0$, we have the unique solution $\mu^N(t)\in X_\gamma(T)$. Moreover, \Cref{prop:Pwellposed} says that $\mu^N(t)$ can be represented as
		\[
		\mu^N(t) = \Phi_{\mu^N}(t,\cdot)\#\mu_0^N,
		\]
		where $\Phi_{\mu^N}$ is the (unique!) flow map in~\eqref{eq:contODE} induced by the curve $\mu^N$. Since $\mu^N$ is the push-forward of $\mu_0^N$, it is also an empirical measure with the form
		\[
		\mu^N(t) = \sum_{i=1}^N m_{i,N}\delta_{\Phi_{\mu^N}(t,v_0^{i,N})}.
		\]
		Moreover, for every $i=1,\dots, N$, $\Phi_{\mu^N}(t,v_0^{i,N})$ solves precisely~\eqref{eq:epsparticle}. 
	\end{proof}
	The following proposition gives a lower bound on the minimum inter-particle distance
	\[
	\eta_m^N(t) = \min_{i,j=1,\dots,N} |v^i(t) - v^j(t)|.
	\]
	\begin{proposition}[No collisions in finite time]
		\label{prop:nocollisionssoft}
		Fix $\varepsilon,\, T>0, \,\gamma\in(-2,0]$, and $\eta_m^N(0) = \min_{i\neq j}|v_0^i - v_0^j|>0$. Then, there is a constant $C = C(\varepsilon, T, M_2(\mu_0^N))>0$ such that the minimum inter-particle distance decays with exponential rate
		\[
		\eta_m^N(t) \gtrsim \eta_m^N(0)\exp\{-Ct\}, \quad \forall t\in[0,T].
		\]
	\end{proposition}
	\begin{proof}
		Choose two indices $i, \, j=1,\dots, N$ such that $\eta_m^N = |v^i - v^j|$ where we will suppress the time dependence for simplicity. We have
		\begin{align*}
			\frac{\rd}{\rd t}|v^i-v^j| &\ge -|U^\varepsilon[\mu^N](v^i) - U^\varepsilon[\mu^N](v^j)| \\
			&\ge -\int_{\R^3}|K_{\mu^N}(v^i, w) - K_{\mu^N}(v^j, w)| d\mu^N(w).
		\end{align*}
		The goal is to estimate the integral. Firstly, we simplify the integration by recalling that supp$\mu^N$ is bounded. Indeed, setting $R = \max_{i=1,\dots,N}|v_0^i|$, \Cref{lem:growthPhi} implies
		\[
		\japangle{v^i} = \japangle{\Phi_{\mu^N}(v_0^i)}\le R \exp (C_\varepsilon M_2(\mu_0^N)t), \quad \forall t\in[0,T].
		\]
		Applying~\Cref{lem:contourestimate} to the difference of the kernels, we have
		\begin{align*}
			&\quad \int_{\text{supp}\mu^N}|K_{\mu^N}(v^i,w) - K_{\mu^N}(v^j,w)| d\mu^N(w) \\ &\lesssim_{\varepsilon,\gamma}|v^i-v^j|\int_{\text{supp}\mu^N}\max\left(
			|v^i-w|^{2+\gamma}, |v^j - w|^{2+\gamma}
			\right)d\mu^N(w)   \\
			&\lesssim_\gamma R\exp(C_\varepsilon M_2(\mu_0^N) T)|v^i-v^j|.
		\end{align*}
	\end{proof}
	
	In the case $\gamma\in(-3,-2)$, we can no longer apply~\Cref{thm:existmf} directly, since it requires an $L^p$ assumption on the initial data $\mu_0^N$ which is not valid for empirical measures. In particular, the vector field is no longer Lipschitz regular (c.f.~\Cref{prop:ULipLp}) so we must make do with H\"older regularity (c.f.~\Cref{lem:UHold}).
	\begin{proof}[Proof of~\Cref{lem:existparticle} for $-3< \gamma< -2$]
		We revisit the proof of Peano's theorem using Schauder's fixed point theorem to construct solutions to~\eqref{eq:epscty}. We set $X = C([0,T];\R^d)$ and define the solution map $S: X \to X$ by
		\[
		(Sv^i)(t) := v_0^i + \int_0^t U^\varepsilon[\mu^N(s)](v^i(s))\, \rd s, \quad i=1,\dots,N.
		\]
		This is well-defined and certainly $Sv^i\in X$ for each $v^i\in X$ owing to the uniform bound for $U^\varepsilon$ in~\Cref{prop:boundsKg} when $\gamma\in(-3,-2)$. We seek to prove 1) $S$ is continuous and 2) $S(X)$ is pre-compact.
		
		\ul{$S$ is continuous:} For every $i=1,\dots,N$ fix $v^{i,n}, \, v^i\in X$ such that $v^{i,n}\to v^i$ in $X$. We label their corresponding empirical measures
		\[
		\mu^N(t) = \sum_{i=1}^Nm_i \delta_{v^i(t)}, \quad \mu^{N,n}(t) = \sum_{i=1}^N m_i \delta_{v^{i,n}(t)}.
		\]
		We have the estimate
		\begin{align*}
			&\quad    |(Sv^{i,n})(t) - (Sv^i)(t)| \le \int_0^t\left|U^\varepsilon[\mu^{N,n}(s)](v^{i,n}(s)) - U^\varepsilon[\mu^N(s)](v^i(s))\right|\rd s     \\
			&\le \int_0^t\left|U^\varepsilon[\mu^{N,n}(s)](v^{i,n}(s)) - U^\varepsilon[\mu^{N,n}(s)](v^i(s))\right|  \\
			&\qquad \qquad \qquad \qquad \qquad \qquad + \left|U^\varepsilon[\mu^{N,n}(s)](v^i(s)) - U^\varepsilon[\mu^N(s)](v^i(s))\right| \rd s.
		\end{align*}
		Applying~\Cref{lem:UHold} to the first difference and~\Cref{lem:gfvel} to the second difference without being precise about the constants, we obtain
		\begin{align*}
			&\quad |(Sv^{i,n})(t) - (Sv^i)(t)| \lesssim_{\varepsilon,\gamma} \int_0^t |v^{i,n}(s) - v^i(s)|^{3+\gamma} \rd s \\
			&\quad +\int_0^t W_\infty(\mu^{N,n}(s), \mu^N(s)) + W_\infty(\mu^{N,n}(s), \mu^N(s))^{3+\gamma}\rd s.
		\end{align*}
		The first integral converges to 0 as $n\to \infty$. As well, the infinite Wasserstein distance is also continuous with respect to the particles; $v^{i,n}\to v^i$ in $X$ for every $i=1,\dots,N$ as $n\to \infty$ implies $W_\infty(\mu^{N,n}(s), \mu^N(s)) \to 0$ as $n\to \infty$.
		
		\ul{$S(X)$ is pre-compact:} We fix $v^i\in X$ for every $i=1,\dots,N$ in this step. Firstly, it is clear that $S(X)$ is bounded using~\Cref{prop:boundsKg}
		\[
		|(Sv)(t)|\le |v(0)| + C_\varepsilon t, \quad \forall v\in X.
		\]
		Turning to equicontinuity, fix $t_1\le t_2$ both in $[0,T]$. Applying~\Cref{prop:boundsKg} again, we have 
		\[
		|(Sv^i)(t_1) - (Sv^i)(t_2)| \le \int_{t_1}^{t_2}|U^\varepsilon[\mu^N(s)](v^i(s))|\, \rd s \lesssim_\varepsilon |t_1-t_2|.
		\]
	\end{proof}
	\section{Step 3}
	\label{sec:step3}
	In this appendix, we prove step~\ref{step3} from~\Cref{sec:mfl} which establishes~\Cref{thm:CCH}. The results of~\Cref{sec:step1,sec:step2} yield
	\begin{align}
		\label{eq:etacouple}
		\begin{split}
			\frac{\rd}{\rd t}\eta^N &\lesssim_{\varepsilon} \eta^N (1 + \|f \|_{L^p})(1 + \left(\eta^N\right)^\frac{d}{p'}\left(\eta_m^N\right)^{1+\gamma}), \\
			\frac{\rd}{\rd t}\eta_m^N &\gtrsim_\varepsilon - \eta_m^N (1 + \| f\|_{L^p})(1 + \left(\eta^N\right)^\frac{d}{p'}\left(\eta_m^N\right)^{1+\gamma}),
		\end{split}
	\end{align}
	when $t\in[0, \min (T_m,T^N)).$ If $\left(\eta^N\right)^\frac{d}{p'}\left(\eta_m^N\right)^{1+\gamma} \le 1$, then we immediately obtain
	\begin{equation}
		\label{eq:growdecay}
		\eta^N(t) \le \eta^N(0) e^{C(1 + \| f\|_{L^p})t}, \quad \eta_m^N(t) \ge \eta_m^N(0) e^{-C(1 + \| f\|_{L^p})t}, \quad \forall t\in[0, \min(T_m,T^N)).
	\end{equation}
	We wish to show that~\eqref{eq:growdecay} holds for all $t\in[0,T_m)$ as $N\to \infty$ which amounts to showing $T^N > T_m$ when $N$ is sufficiently large. Define first
	\[
	a(t) := \frac{\eta^N(t)}{\eta^N(0)}, \quad \eta_m(t) := \frac{\eta_m^N(t)}{\eta_m^N(0)}, \quad \xi_N := \eta^N(0)^\frac{d}{p'}\eta_m^N(0)^{1+\gamma}.
	\]
	Thus, we rewrite~\eqref{eq:etacouple} in terms of $a, \, b, $ and $\xi_N$
	\begin{align*}
		\frac{\rd}{\rd t}a &\lesssim_{\varepsilon} a (1 + \|f \|_{L^p})(1 + \xi_N a^\frac{d}{p'}b^{1+\gamma}), \\
		\frac{\rd}{\rd t}b &\gtrsim_\varepsilon - b (1 + \| f\|_{L^p})(1 +\xi_N a^\frac{d}{p'}b^{1+\gamma}).
	\end{align*}
	Since $a(0) = b(0) = 1$ and we assume by~\eqref{eq:initcond} $\xi_N \to 0$ as $N\to \infty$, when $N$ is sufficiently large, we can find $T_*^N (\le T^N)$ such that
	\begin{equation}
		\label{eq:bddby1}
		\xi_N a^\frac{d}{p'}b^{1+\gamma} \le 1, \quad \forall t\in[0,T_*^N].
	\end{equation}
	Now by~\eqref{eq:growdecay}, we have similar estimates
	\[
	a(t) \le e^{C(1+\|f\|_{L^p})t}, \quad b(t) \ge e^{-C(1+\|f\|_{L^p})t}, \quad \forall t\in[0,T_*^N].
	\]
	Returning to~\eqref{eq:bddby1}, we obtain an estimate for $T_*^N$ given by
	\[
	\xi_N e^{C(1+\|f\|_{L^p})\left(\frac{d}{p'} - (1+\gamma)\right)t} \le 1 \iff t \le - \frac{\log \xi_N}{C(1+\|f\|_{L^p})\left(\frac{d}{p'} - (1+\gamma)\right)}.
	\]
	This means that $T_*^N$ has the lower bound
	\[
	- \frac{\log \xi_N}{C(1+\|f\|_{L^p})\left(\frac{d}{p'} - (1+\gamma)\right)} \le T_*^N.
	\]
	However, since~\eqref{eq:initcond} means $\xi_N \to 0$ as $N\to \infty$, this implies
	\[
	\liminf_{N\to \infty}T_*^N = \infty.
	\]
	Since $T_*^N < T^N$, we have that $T^N \ge T_m$ for $N\gg 1$ sufficiently large.
\end{appendices}
\bibliographystyle{abbrv}
\bibliography{refs}
\end{document}